\documentclass{template_arxiv}

\usepackage{amsmath,amsfonts,amssymb,amsthm}
\usepackage{graphicx}
\usepackage{xcolor}
\usepackage{enumitem,hyperref,bm}

\usepackage{tikz} 
\usetikzlibrary{arrows.meta}
\usetikzlibrary{arrows}
\usetikzlibrary{positioning}

\usepackage[T1]{fontenc}

\usepackage[ruled,vlined]{algorithm2e} 

\newtheorem{theorem}{Theorem}[section]
\newtheorem{conjecture}[theorem]{Conjecture}
\newtheorem{corollary}[theorem]{Corollary}

\newtheorem{example}[theorem]{Example}
\newtheorem{lemma}[theorem]{Lemma}
\newtheorem{proposition}[theorem]{Proposition}
\newtheorem{question}[theorem]{Question}

\newtheorem{remark}[theorem]{Remark}

\numberwithin{equation}{section}

\newcommand{\N}{\mathbb N}
\newcommand{\R}{\mathbb R}
\newcommand{\pref}{\mathrm{pref}}
\newcommand{\lan}{\mathcal{L}}
\newcommand{\eps}{\epsilon}

\newcommand{\ds}{\displaystyle}

\newcommand{\rt}{\mathrm{rt}} 
\newcommand{\wmax}{\mathtt{wmax}}
\newcommand{\wmin}{\mathtt{wmin}}

\newcommand{\nmin}{\mathrm{nmin}}
\newcommand{\val}{\mathrm{val}}
\newcommand{\rep}{\mathtt{rep}}
\newcommand{\card}{\mathrm{Card}}
\newcommand{\rad}{<_{\mathrm{radix}}}
\newcommand{\tree}{\mathcal{T}} 

\date{April 05, 2026}

\title{A Normality Conjecture on Rational Base Number Systems}

\author[,1,2]{M\'elodie Andrieu \thanks{melodie.andrieu@univ-littoral.fr}}
\author[,1]{Shalom Eliahou \thanks{eliahou@univ-littoral.fr}}
\author[,1,2]{L\'eo Vivion \thanks{lvivion.math@gmail.com}}

\affil[1]{Laboratoire de Math\'ematiques Pures et Appliqu\'ees Joseph Liouville, Universit\'e du Littoral C\^ote d'Opale, UR 2597, F-62100 Calais, France and CNRS, FR 2037, France.}

\affil[2]{Centro de Modelamiento Matem\'atico, Universidad de Chile and IRL-CNRS 2807, Beauchef 851, Santiago, Chile.}

\newcommand{\shorttitle}{A Normality Conjecture on Rational Base Number Systems}

\newcommand{\shortauthor}{Andrieu, Eliahou, Vivion}

\begin{document}

\maketitle

\begin{abstract}
	The rational base number system, introduced by Akiyama, Frougny, and Sakarovitch in 2008, is a generalization of the classical integer base number system. Within this framework  two interesting families of infinite words emerge, called minimal and maximal words. We conjecture that every minimal and maximal word is normal over an appropriate subalphabet. To support this conjecture, we present extensive numerical experiments that examine the richness threshold and the deviation from normality of these words. We also discuss the implications that the validity of our conjecture would have for several long-standing open problems, including the existence of $Z$-numbers (Mahler, 1968) and $Z_{p/q}$-numbers (Flatto, 1992), the existence of triple expansions in rational base $p/q$ (Akiyama, 2008), and the Collatz-inspired `4/3 problem' (Dubickas and Mossinghoff, 2009).

\vspace*{.5cm}
\end{abstract}

%
%

\section{Introduction}

The rational base number system was first studied by Akiyama, Frougny, and Sakarovitch in 2008.
Given $p > q$ coprime positive integers, the \emph{expansion} of a nonnegative integer $n$ in rational base $p/q$, which we denote by $\rep_{p/q}(n)$, is the unique finite word \[
    a_k a_{k-1}\cdots a_0,
\]
in which the letters $a_i$ belong to the alphabet $\{0,1,\dots,p-1\}$, and such that
\[
    \begin{cases}a_k \neq 0, \\ n=\frac{1}{q}\sum_{i=0}^k a_i \Big(\frac{p}{q}\Big)^i.
    \end{cases}
\]
When the denominator $q$ is $1$, we recover the classical integer bases. However, when $q$ is not equal to $1$, this numeration system exhibits a surprisingly complex behavior. As an illustration, in rational base $7/3$ the expansions of the integers from $0$ to $18$ are:
\begin{equation} \label{eq:enumerate73}
\eps, \mathtt{3}, \mathtt{6}, \mathtt{32}, \mathtt{35}, \mathtt{61}, \mathtt{64}, \mathtt{320}, \mathtt{323}, \mathtt{326}, \mathtt{352}, \mathtt{355}, \mathtt{611}, \mathtt{614}, \mathtt{640}, \mathtt{643}, \mathtt{646}, \mathtt{3202}, \mathtt{3205}, \end{equation}
where $\eps$ denotes the empty word.


In this paper we are interested in two classes of infinite words that naturally emerge within rational base number systems: the classes of minimal and maximal words. They can be defined as follows. Let $u$ be the expansion of an integer in rational base $p/q$. The \emph{minimal} (resp. \emph{maximal}) word with \emph{seed} $u$, denoted by $\wmin_{p/q}(u)$ (resp. $\wmax_{p/q}(u)$), is the unique infinite word that satisfies the following property: for every $l \in \N$, its prefix of length $l$ is the lexicographically minimal (resp. maximal) element $v \in \{0,\ldots,p-1\}^l$ such that the concatenation $uv$ is still the expansion of an integer in rational base $p/q$.

\begin{example}\label{ex:infinite_wmin}
Let $u$ and $v$ be two finite words. We recall that $u<v$ for the \emph{radix order} if the word $u$ is strictly shorter than $v$, or, when they are of the same length, if $u$ is lexicographically smaller than $v$. For instance, $ba<abb<baa$. If we temporarily admit that, with respect to the radix order, the expansion of an integer $n$ increases as $n$ increases, then we deduce from the list \eqref{eq:enumerate73} that $\wmin_{7/3}(\mathtt{3})$ starts with $\mathtt{202}$, and that $\wmax_{7/3}(\mathtt{3})$ starts with $\mathtt{55}$. Further computations would show that
	\[
	\wmin_{7/3}(\mathtt{3})=
	\mathtt{202122220200012011010222102122101011102220120011100201010...}%
	\]
	\[
	\wmax_{7/3}(3)=
	\mathtt{554646446556454454665466654564564564565645445466446666455...}%
	\]

As the reader may observe:\\
- these two words are written over the subalphabets $\{0,1,2\}$ and $\{4,5,6\}$, respectively; \\
- the distribution of their letters seems erratic.\\
(Other examples of infinite minimal and maximal words can be found in the Online Encyclopedia of Integer Sequences; see, for example, the sequence $A304274$, and the referencing work undertaken in \cite{Stu18}).
\end{example}

It is easy to see that when $q=1$, all minimal words coincide and are equal to the infinite word written with the sole letter $0$:
$$ w = 0000000000000000000000000...,$$
while all maximal words are equal to the infinite word
written with the sole letter $p-1$.
By contrast, when $q\neq 1$, all minimal words on the one hand, and all maximal words on the other hand are pairwise distinct, and none of them is eventually periodic \cite[Proposition~26]{AFS08}. Further insights can be gained through the notion of complexity. By definition, the \emph{complexity} of an infinite word $w$ is the function $l \mapsto \mathsf{p}_w(l)$ that counts, for each nonnegative integer $l$, the number of distinct subwords of length $l$ that one reads in $w$. A celebrated theorem by Morse and Hedlund \cite[Theorems~7.3 and 7.4]{MH38} asserts that an infinite word $w$ is not eventually periodic if and only if its complexity satisfies:
\begin{equation}\label{eq:MH38}
    \mathsf{p}_w(l)\geq l+1
\end{equation}
for all lengths $l$.
From this perspective, the complexity captures how far from being periodic---or 
how \emph{chaotic}---an infinite word is.
In this direction, Dubickas \cite[Theorem~3]{Dub09} established that the complexity of every minimal word in rational base $p/q$ satisfies 
\begin{equation}\label{eq:dubickas_complexity}\liminf_{l\rightarrow \infty} \frac{\mathsf{p}_w(l)}{l} \geq \frac{\log q}{\log(p/q)}.
\end{equation}
The latter expression gives another linear lower bound for the complexity of minimal words, which slightly improves \eqref{eq:MH38} when $p<q^2$:
\begin{equation}\label{eq:Dubickas2}\liminf_{l\rightarrow \infty} \frac{\mathsf{p}_w(l)}{l}>1. \end{equation}

 \medskip
 
The main focus of this article is the belief that a much stronger statement holds: we expect that every minimal and maximal word in a rational base has maximal complexity, and is even normal over an appropriate subalphabet. 

\begin{conjecture}\label{conj:normality}
	For all rational bases $p/q$ with $p>q\geq 1$ coprime, and for all integer expansions $u\in\lan_{p/q}$, the infinite word $\wmax_{p/q}(u)$ is normal over the subalphabet $\{p-q,\dots,p-1\}$.  For all integer expansions $u\in\lan_{p/q}$ except for the empty word $\eps$, the infinite word $\wmin_{p/q}(u)$ is normal over the subalphabet $\{0,\dots,q-1\}$. 
\end{conjecture}

\begin{remark}
	(i) The minimal word with seed the empty word must be excluded from Conjecture~\ref{conj:normality} because it starts with the letter $q$. This letter arises from the fact that no expansion is allowed to begin with the letter $0$. If we remove the first letter of $\wmin_{p/q}(\eps)$, we obtain the word $\wmin_{p/q}(q)$, which falls within the scope of the conjecture.
\\
(ii) As we shall see, minimal and maximal words are pairwise related, and the normality of all minimal words but $\wmin_{p/q}(\eps)$ is equivalent to the normality of all maximal words. 
\end{remark}

We recall that an infinite word $w$ is \emph{normal} over a $d$-ary alphabet if for all $l\geq1$, each of the $d^l$ words of length $l$ occurs in $w$ with the same limit frequency $1/d^l$.
The notion of normality, introduced by \'Emile Borel in 1909 to study the distribution of digits in real numbers \cite{Bor09}, has a rich history filled with challenging and unresolved questions (see, for example, the survey \cite{Que06}, \cite[Chapter~4]{BB08}, \cite[Chapters~4--6]{Bug12}, or the lecture notes \cite{BC18}). In rational base $p/q$, normality has been previously studied in \cite{MST13} on the example of the Champernowne word.

\medskip

It is noteworthy that Conjecture~\ref{conj:normality} is trivially true when $q=1$, that is, for integer bases. Indeed, as we saw, all minimal words are equal to the infinite word written with the sole letter $0$, and all maximal words are equal to the infinite word
 written with the sole letter $q-1$, which are normal over the singleton alphabets $\{0\}$ and $\{q-1\}$, respectively. 
 The aim of the paper is to convince the reader that our conjecture seems true and of considerable difficulty when $q\neq 1$.
 
 In the first direction, we present and discuss the results of extensive numerical experiments that support our conjecture. These experiments consist in computing the  \emph{richness threshold} and the \emph{deviation from uniform distribution} of minimal words, and compare them with those of celebrated infinite words known or conjectured to be normal, and with those of random words.
 We investigated in total $139$ minimal words up to length $10^6$, and 40,000 minimal words up to length $10^5$, in various rational bases with $2\leq q < p\leq 26$. 
 
In the second direction, we show that the validity of  Conjecture~\ref{conj:normality} would imply the truth of four long-standing conjectures:
\begin{itemize}
    \item[-] a celebrated conjecture by Mahler from 1968, which asserts the non-existence of `$Z$-numbers' \cite{Mah68};
    
    \item[-] one of its generalizations: the non-existence of `$Z_{p/q}$-numbers', when $p<q^2$ (see Conjecture \ref{conj:mahler_generalized} below);
    
    \item[-] a conjecture by Akiyama from 2008, according to which no real number admits a triple expansion in rational base $p/q$ \cite{Aki08};
    
    \item[-] a conjecture by Dubickas and Mossinghoff from 2009 concerning the termination of certain iterated maps on integers \cite{DM09}.
\end{itemize}

\begin{conjecture}
	\label{conj:mahler_generalized} Let $p>q> 1$ be coprime integers such that  $p< q^2$. There exists no positive real number $x$ (called \emph{$Z_{p/q}$-number}) such that the sequence of fractional parts
	\[(\{x(p/q)^n\})_{n\in \N}\]
	is contained in the subinterval $[0,1/q).$
\end{conjecture}

\begin{theorem}\label{th:impliesMahler} The veracity of Conjecture~\ref{conj:normality} implies that of Conjecture~\ref{conj:mahler_generalized}.
\end{theorem}

As the proofs will show, our normality Conjecture \ref{conj:normality} is significantly stronger than the four aforementioned conjectures. For example, we will see that proving all minimal words contain the letter $0$ at least once---a statement considerably weaker than normality---is already sufficient to establish Akiyama's triple expansion conjecture.

\medskip
Finally,  we show that Conjecture~\ref{conj:normality} can be equivalently formulated in terms of equidistribution in residue classes. We recall that an integer sequence $(u_n)$ is \emph{equidistributed in the residue classes modulo} $m$ if  the frequencies of the events $u_n \equiv r \mod m$, for $r \in \{0,\dots,m-1\}$, are all equal to $1/m$.

\begin{conjecture}\label{conj:equidistribution} Let $p > q \geq 1$ be coprime integers. For every $n \in \N_{>0}$, and for every nonnegative integer $k$, the integer sequence $(T_{p/q}^l(n))_{l \in \N}$, obtained by iterating the operator 
\[ T_{p/q}(x) := \Big\lceil \frac{p}{q}x\Big\rceil,\]
is equidistributed in the residue classes modulo $q^k$. \end{conjecture}

\begin{theorem}\label{th:conj_equivalence}
Conjectures~\ref{conj:normality} and \ref{conj:equidistribution} are equivalent.
\end{theorem}

\medskip

Our work can be understood as a three-direction generalization of a recent article by Eliahou and Verger-Gaugry \cite{EVG25}, which explores the possibility that the letters (and not the finite words) are equidistributed in the maximal word with seed the empty word, in the particular base $p/q=3/2$.

\paragraph{Outline} The article is divided into four sections. In Section~\ref{sect:def_and_construction_minimal words}, we recall the main properties of rational base number systems, and explain how to compute minimal and maximal words. In Section~\ref{sect:links_between_conjectures}, we prove Theorems \ref{th:impliesMahler} and \ref{th:conj_equivalence}, that is, we show that Conjectures~\ref{conj:normality} and \ref{conj:equidistribution} are equivalent, and stronger than Conjecture \ref{conj:mahler_generalized}. We also prove that the validity of our conjecture easily implies the truth of the triple-expansion conjecture by Akiyama from 2008, and the truth of the Collatz-like conjecture by Dubickas and Mossinghoff from 2009. In Section~\ref{sect:numerical_evidence}, we present and discuss the numerical experiments that support our conjecture.


\section{Preliminaries: Rational Base Number System, Minimal and Maximal Words}\label{sect:def_and_construction_minimal words}

The results of this section are taken from the seminal paper \cite{AFS08}.

\subsection{The Modified Division Algorithm}

In an integer base $b$, there exist two algorithms to compute the expansion of an integer. The first one, called the \emph{greedy algorithm}, consists of successively removing the largest possible power of $b$ as many times as possible. It outputs the digits of 
$n$ from most to least significant.
By contrast, the second algorithm computes the least significant digits first. It consists of dividing $n$ by $b$ and repeating the process on the successive quotients until the quotient is zero. 
Rational base number systems are based on a generalization of this second algorithm, called the \emph{Modified Division Algorithm} (MD algorithm).

\begin{algorithm}
	\caption{Modified Division Algorithm}
	\KwIn{Integer $n$}
	\KwOut{A finite word $(a_k,a_{k-1},\dots,a_0)$}
	
	$k \leftarrow 0$\;   
	\While{$n > 0$}{
		\# the Euclidean division of $qn$ by $p$\\
		$a_k \leftarrow (q \cdot n) \bmod p$\;
		$n \leftarrow \left\lfloor \dfrac{q \cdot n}{p} \right\rfloor$\;   
		$k \leftarrow k + 1$\;
	}
	\Return $(a_k, a_{k-1}\dots,a_0)$\;	
\end{algorithm}

The output of the MD algorithm is called the \emph{expansion} of $n$ in the rational base $p/q$, and denoted by $\rep_{p/q}(n)$. 
 One can retrieve $n$ from its expansion using the (modified) evaluation function
\begin{equation}\label{eq:def_val}
\val_{p/q}(a_k\cdots a_0):=\frac{1}{q}\sum_{i=0}^k a_i \Big(\frac{p}{q}\Big)^i.
\end{equation}
Furthermore, it can be shown that $\rep_{p/q}(n)$ is the unique finite word $u \in \{0, \dots, p-1\}^*$ that does not start with $0$ and whose evaluation yields $n$.

\medskip

\begin{example}\label{ex:MD_algo} The expansions of $0$ and $1$ in any rational base $p/q$ are, respectively, the empty word $\eps$ and the single-digit word $q$. The expansion of $10$ in the rational base $7/3$ is calculated with the following steps:
\[
 30 = 4 \times 7 + 2, \quad
12 = 1 \times 7 + 5, \quad
3 = 0 \times 7 + 3.
\]
This yields $\rep_{7/3}(10) = \mathtt{352}$. When $q=1$, the expansion in the rational base $p/q$ coincides with the expansion in the integer base $p$.
\end{example}

\subsection{The language of integers $\lan_{p/q}$}

In this paper, we are interested in the \emph{language of integers} defined by
\[\lan_{p/q} = \{ \rep_{p/q}(n) : n \in \mathbb{N} \}.\]
This language $\lan_{p/q}$ is known to exhibit highly complex and non-regular behavior when $q \neq 1$ (for example, it is not context-free). However, it satisfies two fundamental properties: it is \emph{prefix-closed} and \emph{right-extendable}.

\begin{proposition}\label{prop:closeness_and_extendability} Let $p > q \geq 1$ be coprime integers. Then:\\
	1. The language $\lan_{p/q}$ is prefix-closed; that is, every prefix of a word in $\lan_{p/q}$ also belongs to $\lan_{p/q}$.\\
	2. The language $\lan_{p/q}$ is right-extendable: for every $u \in \lan_{p/q}$, there exists a letter $a \in \{0, \dots, p-1\}$ such that $ua \in \lan_{p/q}$. More precisely, if $u \in \lan_{p/q}$, then 
	\[
	ua \in \lan_{p/q} \iff ua \neq \mathtt{0} \text{ and } a = -p \cdot \val_{p/q}(u) \mod{q}.
	\] 
\end{proposition}

\begin{remark}
	By definition, no expansion begins with the letter $0$. Therefore, in Proposition~\ref{prop:closeness_and_extendability} (Assertion 2), we must explicitly exclude $0$ from the set of extensions of the empty word. This explains the condition $ua \neq \mathtt{0}$ in the equivalence. 
\end{remark}

\begin{proof} Let $u \in \lan_{p/q}$.\\ 
	1. Let $l$ be the length of $u$. Applying the first step of the MD algorithm shows that the prefix of length $l - 1$ of $u$ is the expansion of the integer
	$\left\lfloor qn/p \right\rfloor$
	in  base $p/q$.\\
	2. Let $n = \val_{p/q}(u)$. A letter $a \in \{0,\ldots,p-1\}$ extends $u$ if and only if $\val_{p/q}(ua)$ is an integer and $ua \neq \mathtt{0}$. Since 
	\[
	\val_{p/q}(ua) = \frac{p}{q} n + \frac{a}{q},
	\]
	this holds if and only if $a = -p n \mod{q}$ and $ua \neq \mathtt{0}$.
	Furthermore, since $q < p$, such a letter $a$ always exists in the alphabet $\{0, \dots, p-1\}$.
\end{proof}

\begin{corollary}\label{cor:csq_MDalgo}
	1. The smallest letter extending the empty word is always $q$, while the smallest letter extending any other  $u  \in \lan_{p/q}$ belongs to $\{0,\ldots,q-1\}$. The largest letter extending any $u  \in \lan_{p/q}$ belongs to $\{p-q+1,\ldots,p-1\}$. \\
	2. The set of letters extending a non-empty word $u$  depends only on the residue class of $\mathrm{val}_{p/q}(u)$ modulo $q$.\\
3. More generally, for every $l \in\N$, the set of finite words $v$ of length $l$ that extend a non-empty word $u$ depends only on the residue class of $\val_{p/q}(u)$ modulo $q^l$.\end{corollary}

\subsection{The tree $\tree_{p/q}$ associated with the language of integers }

Since $\lan_{p/q}$ is prefix-closed and right-extendable, it can be represented as an infinite directed tree. This tree plays a major role in \cite{AFS08}, as it is closely related to the representation of real numbers (and not only integers) in rational base. The tree structure of rational base number systems has been further studied in \cite{MS17,AMS18}.
\medskip

In base $p/q$, the tree $\tree_{p/q}$ associated with the language of integers is defined as follows:
\begin{itemize}
	\item The nodes are the words $u \in \lan_{p/q}$, each labeled by both $u$ and its valuation $n \in \mathbb{N}$.
	\item There is an edge labeled by a letter $a \in \{0, \dots, p-1\}$ from $u$ to $v$ if $v = ua$.
\end{itemize}
When representing the tree $\tree_{p/q}$, edges are oriented from left to right, and outgoing edges are ordered from bottom to top according to the natural order on $\{0, \dots, p-1\}$.

\begin{example} The tree associated with the language $\lan_{7/3}$ is shown in Figure~\ref{fig:tree}. When $q = 1$, that is, in the classical integer base case, the language $\lan_{p/q}$ consists of all finite $p$-ary words that do not begin with $0$. It is therefore represented by the complete $p$-ary tree, from which  the node $0$ and the subtree rooted at this node are removed.
\end{example}


\newcommand{\twolines}[2]{%
	\texttt{#1}\\[-2pt]\scriptsize \textcolor{black!60}{$\mathit{#2}$}
}
\tikzset{
	vertex/.style={
		circle,
		draw,
		align=center,
		inner sep=2pt,
		minimum size=9mm
	}
}

\begin{figure}[!h]
	\begin{center}
		\footnotesize
		\begin{tikzpicture}[scale=0.83]
		
		\begin{scope}[every node/.style={vertex}]
		\node (A) at (0,0) {\twolines{$\epsilon$}{0}};
		
		\node (B) at (2,-3) {\twolines{3}{1}};
		\node (C) at (2,3) {\twolines{6}{2}};
		
		\node (D) at (4.5,-4.5) {\twolines{32}{3}};
		\node (E) at (4.5,-1.5) {\twolines{35}{4}};
		\node (F) at (4.5,1.5) {\twolines{61}{5}};
		\node (G) at (4.5,4.5) {\twolines{64}{6}};
		
		\node (H) at (8,-5.8)   {\twolines{320}{7}};
		\node (I) at (8,-4.5) {\twolines{323}{8}};
		\node (J) at (8,-3.2) {\twolines{326}{9}};
		\node (K) at (8,-2) {\twolines{352}{10}};
		\node (L) at (8,-0.8) {\twolines{355}{11}};
		\node (M) at (8,0.8)   {\twolines{611}{12}};
		\node (N) at (8,2) {\twolines{614}{13}};
		\node (O) at (8,3.2) {\twolines{640}{14}};
		\node (P) at (8,4.5) {\twolines{643}{15}};
		\node (Q) at (8,5.8) {\twolines{646}{16}};
		
		\node (R) at (12,-6.5) {\twolines{3202}{17}};
		\node (S) at (12,-5.1)  {\twolines{3205}{18}};
		\end{scope}
		
		\begin{scope}[
		>={Stealth[black]},
		every node/.style={fill=white,circle},
		every edge/.style={draw=black!30, thick}
		]
		\path [->] (A) edge node {\texttt{3}} (B);
		\path [->] (A) edge node {\texttt{6}} (C);
		
		\path [->] (B) edge node {\texttt{2}} (D);
		\path [->] (B) edge node {\texttt{5}} (E);
		\path [->] (C) edge node {\texttt{1}} (F);
		\path [->] (C) edge node {\texttt{4}} (G);
		
		\path [->] (D) edge node {\texttt{0}} (H);
		\path [->] (D) edge node {\texttt{3}} (I);
		\path [->] (D) edge node {\texttt{6}} (J);
		\path [->] (E) edge node {\texttt{2}} (K);
		\path [->] (E) edge node {\texttt{5}} (L);
		\path [->] (F) edge node {\texttt{1}} (M);
		\path [->] (F) edge node {\texttt{4}} (N);
		\path [->] (G) edge node {\texttt{0}} (O);
		\path [->] (G) edge node {\texttt{3}} (P);
		\path [->] (G) edge node {\texttt{6}} (Q);
		
		\path [->] (H) edge node {\texttt{2}} (R);
		\path [->] (H) edge node {\texttt{5}} (S);
		\end{scope}
		
		\foreach \y in {-4.4, -3.1, -1.9, -0.7, 0.9, 2.1, 3.3, 4.6, 5.9} {
			\draw (8.8,\y) node[anchor=north west] {$\bm{\ldots}$};
		}
		
		\foreach \y in {-6.4, -5} {
			\draw (12.8,\y) node[anchor=north west] {$\bm{\ldots}$};
		}
		
		\end{tikzpicture}
		\normalsize
		\caption{\label{fig:tree} The tree $\mathcal{T}_{7/3}$ representing the language of integers in rational base $7/3$. The nodes are  labeled simultaneously by $u \in \lan_{p/q}$ (above)  and $n=\val_{p/q}(u) \in \N$ (below).}
	\end{center}
\end{figure}

As shown in Figure~\ref{fig:tree}, a breadth-first traversal of $\tree_{7/3}$, in which nodes at the same depth are read from bottom to top, corresponds simply to counting. This property holds for every tree $\tree_{p/q}$ associated with a rational base number system. It can be viewed as a consequence of the next proposition, together with the fact that, by construction of $\tree_{p/q}$, a breadth-first traversal enumerates the expansions of integers with respect to the radix order. 

\begin{proposition}\label{prop:radix_increasing}The map
$\val_{p/q} : \lan_{p/q} \to \mathbb{N}$
is (bijective and) increasing with respect to the radix order.
\end{proposition}

Proposition~\ref{prop:radix_increasing} can be proved by induction on the length of the expansions, using the MD algorithm. It is worth noting that the general valuation map
\[\val_{p/q} : \{0, \dots, p-1\}^* \to \mathbb{Q}\]
is not monotone. For example, we have $\mathtt{3}\rad \mathtt{10}$, and yet, numerically, $\val_{7/3}(\mathtt{3})=1 > \val_{7/3}(\mathtt{10})= 7/9.$ (In fact, one of the objectives of \cite{AFS08} was precisely to introduce a representation of real numbers in rational base $p/q$ that respects the usual order on $\mathbb{R}$.) 

\bigskip

A second property of the trees $\tree_{p/q}$, already visible in Figure~\ref{fig:tree}, is their periodic structure. More precisely, if we perform a breadth-first traversal of the line graph of $\tree_{7/3}$---that is, if we read the edge labels  by increasing depth, and for edges  at the same depth from bottom to top---we obtain a periodic sequence with period $(3,6,2,5,1,4,0)$. 

\begin{proposition}\label{prop:rythm_signature} A breadth-first traversal \emph{of the edges} of $\tree_{p/q}$ yields the $p$-periodic sequence $(a_k)_{k\geq0}$ in $\{0,\ldots,p-1\}^{\N}$, defined by $a_0=q$ and \[a_{k+1}=a_k+q \mod p\] for every  $k \in \N$.
\end{proposition}

\begin{proof} Let $a, b \in \{0,\ldots,p-1\}$ be two consecutive edges in the breadth-first traversal of the line graph of $\tree_{p/q}$. If $a$ and $b$ originate from the same node $u$, then by Proposition~\ref{prop:closeness_and_extendability}, we have $b = a + q$.
	Otherwise, $a$ and $b$ originate from two consecutive nodes $u$ and $v$. In this case we have
	\[
	\val_{p/q}(vb) = \val_{p/q}(ua) + 1
	\quad \text{and} \quad
	\val_{p/q}(v) = \val_{p/q}(u) + 1,
	\]
	which implies
	\[
	\frac{p}{q}(\val_{p/q}(u)+1) + \frac{b}{q}
	=
	\frac{p}{q}\val_{p/q}(u) + \frac{a}{q} + 1.
	\]
	Hence $b = a + q - p$.
To conclude, note  that the first edge $a_0$ connects the node $0$, whose expansion is the empty word, to the node $1$, encoded by $q$. Therefore $a_0=q$.
\end{proof}

In \cite{MS17}, the periodic structure of the rational base trees $\tree_{p/q}$ is described using the notions of \emph{rhythm} and \emph{signature}.

\subsection{Minimal and maximal words}

Let $u \in \lan_{p/q}$. The \emph{minimal} (resp.\ \emph{maximal}) word with seed $u$ is the infinite word labeling the infinite path starting from the node $u \in \tree_{p/q}$, in which the smallest (resp.\ largest) edge is chosen at each branching node. We denote them by $\wmin(u)$ and $\wmax(u)$, respectively. By abuse of notation, it will sometimes be convenient to write them by $\wmin(n)$ and $\wmax(n)$, where $n \in \mathbb{N}$ is the valuation of $u$.

As in the Introduction, minimal and maximal words can equivalently be defined as follows: for each length $n \in \mathbb{N}$, their prefix of length $l$ is lexicographically minimal (resp.\ maximal) among all words of length $l$ that extend $u$ to an element in $\lan_{p/q}$.

\begin{example}\label{ex:wmin73eps} We read on Figure~\ref{fig:tree} that $\wmin_{7/3}(\eps)$ begins with $\mathtt{3202}$. Further computation shows that
\[\wmin_{7/3}(\eps)=
\mathtt{3202122220200012011010222102122101011102220120011100201010}...\]
Note that the presence of the letter $\mathtt{3}$ at the beginning of $\wmin_{7/3}(\eps)$ is exceptional. It arises from the fact that no expansion of an integer may begin with the digit $0$, which prevents the existence of an outgoing edge labeled $0$ from the node $\eps$.	
\end{example}

Interestingly, although rational bases were not yet formalized at that time, minimal words were first introduced by Mahler in the 1960s in connection with the study of $Z$-numbers (see Section~\ref{sssect:implies_mahler} for details). In~\cite{AFS08,Aki08}, it is shown that minimal and maximal words are related to multiple expansions of real numbers in rational base (see Section~\ref{sssect:linkwithconjectures}). In~\cite{AMS18}, they are studied under the name \emph{bottom} and \emph{top} words.

\medskip
As observed in Examples~\ref{ex:infinite_wmin} and \ref{ex:wmin73eps}, all minimal and maximal words---except $\wmin_{p/q}(\eps)$---are written over the subalphabets $\{0,\dots,q-1\}$ and $\{p-q+1,\dots,p-1\}$, respectively. This follows from Corollary~\ref{cor:csq_MDalgo}.

Furthermore, by Proposition~\ref{prop:rythm_signature}, minimal and maximal words are pairwise related. More precisely, for every $n\in \N$, the maximal word with seed $n$ becomes equal to the minimal word with seed $n+1$ after renaming its letters according to the substitution
\[\begin{array}{lcll}
\sigma : &\{p-q+1,\dots,p-1\}& \to &\{0,\dots,q-1\}, \\
  &a &\mapsto &a+q-p.
  \end{array}
\]

\begin{example}\label{ex:lien_min_max}
Since $\rep_{7/3}(0)=\eps$ and $\rep_{7/3}(1)=\mathtt{3}$, if we know that
\[
\wmax_{7/3}(\eps)=
\mathtt{646566664644456455454666546566545455546664564455544645454...,}
\]
then, by replacing $4, 5, 6$ with $0, 1, 2$, respectively, we obtain
	\[
	\wmin_{7/3}(\mathtt{3})=
	\mathtt{202122220200012011010222102122101011102220120011100201010...}
	\]
\end{example}

\begin{corollary}\label{cor:conj1_redundant}
The two statements in Conjecture~\ref{conj:normality} are equivalent.
\end{corollary}
Therefore, throughout most of the paper, we work with minimal words. This preference is justified by their connection with the operator 
\[T_{p/q}=\Big\lceil \frac{p}{q} \bm{\cdot}\Big\rceil,\]
 which appears in Conjecture~\ref{conj:equidistribution}.

\begin{question} In \cite[Problem~71]{AMS18}, the authors ask whether all minimal words in rational base $p/q$ are of the ``same kind'', that is, whether they can be transformed into one another by a finite-state machine. To the best of our knowledge, this question remains open at the time of writing.
\end{question}



\section{Links between conjectures}\label{sect:links_between_conjectures}

In this section, we prove Theorem~\ref{th:conj_equivalence}, which asserts that our Conjectures~\ref{conj:normality} and \ref{conj:equidistribution} are equivalent. We further prove Theorem~\ref{th:impliesMahler} , which asserts that our Conjecture~\ref{conj:normality} is stronger than that of Mahler (1968) in the general case where $p$ and $q$ are coprime integers satisfying $1<q<p<q^2$. We also discuss the connections between our conjectures and three other earlier conjectures: the first by Akiyama (2008), the second by Mossinghoff and Dubickas (2009), and the third, the famous Collatz conjecture. 

\subsection{Proof of Theorem~\ref{th:conj_equivalence}}

We start with a lemma that connects minimal words in rational base $p/q$ with the operator $T_{p/q}=\lceil \frac{p}{q} \bm{\cdot}\rceil$.

\begin{lemma}\label{lemma:link_wmin_operatorT}
	Let $p>q\geq1$ be two coprime integers. Let $u$ be a non-empty word in $\lan_{p/q}$ and $l$ a nonnegative integer. Denote by $\wmin_{p/q}(u,l)$ the word formed by the first $l$ letters of $\wmin_{p/q}(u)$, and by $\nmin_{p/q}(u,l)$ the valuation of the expansion $u\cdot\wmin_{p/q}(u,l)$, where the symbol $\cdot$ denotes the concatenation:
\[\nmin_{p/q}(u,l) := \val_{p/q}(u\cdot\wmin_{p/q}(u,l)).\]
	Then we have
		\[ \val_{p/q}(u\cdot\wmin_{p/q}(u,l)) = T^l_{p/q}(\val_{p/q}(u)).\]
\end{lemma}

\begin{example}
	Continuing our example in base $7/3$, one reads in Figure~\ref{fig:tree} that 
	\[ \begin{cases}
	\nmin_{7/3}(\mathtt{3},1) = 3,\\
	\nmin_{7/3}(\mathtt{3},2) = 7,\\
	\nmin_{7/3}(\mathtt{3},3) = 17 ;
	\end{cases}\]
	which indeed coincides with the first three iterations of $T_{7/3}$ on $\val_{7/3}(\mathtt{3}) =1$.    
\end{example}

\begin{proof}[Proof of Lemma~\ref{lemma:link_wmin_operatorT}]
	Let $l \in \N$. It follows from Proposition~\ref{prop:closeness_and_extendability} that	
	\[\nmin_{p/q}(u,l+1) = \frac{p}{q} \nmin_{p/q}(u,l) +  \frac{a}{q}, \]
	where $a$ is the smallest nonnegative integer such that $a=-p\cdot\nmin_{p/q}(u,l) \mod q$. Hence
	\[\nmin_{p/q}(u,l+1) = \Big\lceil \frac{p}{q} \nmin_{p/q}(u,l)\Big\rceil. \]
	We conclude by noting that $\nmin_{p/q}(u,0)=\val_{p/q}(u)$.
\end{proof}

\begin{remark}\label{rk:how_to_compute_min_words}
	Still under the condition $u\neq\eps$, the letter $a$ calculated in the proof of Lemma~\ref{lemma:link_wmin_operatorT}, that is, the smallest nonnegative integer such that
	\[a = -p\cdot\nmin_{p/q}(u,l) \mod q\] is the $(l+1)$-th letter of the minimal word $\wmin_{p/q}(u)$.
\end{remark}


\bigskip
We are now in a position to prove  Theorem~\ref{th:conj_equivalence}, that is, that our Conjectures~\ref{conj:normality} and \ref{conj:equidistribution} are equivalent.

\begin{proof}[Proof of Theorem~\ref{th:conj_equivalence}] Let
	\[ \begin{array}{llll}
	f : & \mathbb{Z}/q^l\mathbb{Z} & \longrightarrow & \{0,\dots,q-1\}^l \\
	& n & \longmapsto & \wmin_{p/q}(\rep_{p/q}(n),l).
	\end{array}\]
By Corollary~\ref{cor:csq_MDalgo}, the function $f$ is well-defined and bijective (see also \cite[Lemme~4.14]{Mar16}).
	Let $n \in \N_{>0}$ be an integer, and denote by $u$ its expansion in rational base $p/q$.
	We proceed by equivalences. For readability, the indices $p/q$ are omitted.

	\bigskip
	
	\noindent For every $l\in \N$, the sequence $(T^m(n))_m$ is equidistributed in the residue classes modulo $q^l$ 
	\small
	\[
	\begin{array}{cl}
	\iff & \text{\footnotesize $\forall l\in\N,$ $\forall r \in \{0,\!...,q^l-1\}$, \small } \; \underset{N\to\infty}{\lim}  \frac{\card \big\{m\in \{0,\!...,N\!-\!1\} \;|\; T^m(n) \equiv r \bmod q^l\big\}}{N} = \frac{1}{q^l}\\ \\
	\underset{\text{Lemma~\ref{lemma:link_wmin_operatorT}}}{\iff} & \text{\footnotesize $\forall l\in\N,$ $\forall r \in \{0,\!...,q^l-1\}$, \small } \;\underset{N\to\infty}{\lim} \frac{\card \big\{m\in \{0,\!...,N\!-\!1\} \;|\; \nmin(u,m)\equiv r \bmod q^l\big\}}{N}=\frac{1}{q^l} \\ \\
	{\iff} & \text{\footnotesize $\forall l\in\N,$ $\forall r \in \{0,\!...,q^l-1\}$, \small } \;\underset{N\to\infty}{\lim} \frac{\card \big\{m\in \{0,\!...,N\!-\!1\} \;|\; \wmin(\rep(\nmin(u,m)),l)=f(r)\big\}}{N}=\frac{1}{q^l}
	\end{array}
	\]
	\normalsize
	
	\bigskip
	\noindent (In the last equivalence, we simply used the bijective function $f$.) At this point, it is useful to observe that, by definitions of $\nmin$ and $\wmin$:
	\[
	\wmin(\rep(\nmin(u,m)),l) = \wmin(u.\wmin(u,m),l) = \wmin(u)[m+1:m+l],
	\]
	where $\wmin(u)[m+1:m+l]$ denotes the subword of length $l$ of $\wmin(u)$ spanning from its $(m+1)$-th to its $(m+l)$-th letter.
	
	\bigskip
	With this simplification, we pursue the chain of equivalence:
	\small
	\[
	\begin{array}{cl}
	{\iff} & \text{\footnotesize $\forall l\in\N,$ $\forall v \in \{0,\!...,q-1\}^l$, \small } \;\underset{N\to\infty}{\lim} \frac{\card \big\{m\in \{0,\!...,N\!-\!1\} \;|\; \wmin(u)[m+1:m+l]=v\big\}}{N} = \frac{1}{q^l} \\ \\
	\iff & \text{$\wmin(u)$ is normal over the alphabet $\{0,\dots,q-1\}$.}
	\end{array}
	\]
	\normalsize
	
	We thus proved that the equidistribution of all sequences $(T^m(n))_m$, for $n \in \N\setminus\{0\}$, is equivalent to the normality of all minimal words $\wmin(u)$, for $u\in \lan_{p/q}\setminus\{\eps\}$. 
	By using Corollary~\ref{cor:conj1_redundant} to treat the question of maximal words, the proof of Theorem~\ref{th:conj_equivalence} is complete.
\end{proof}


\subsection{Proof of Theorem~\ref{th:impliesMahler}} \label{sssect:implies_mahler}

In this section, we show that the validity of our normality Conjecture~\ref{conj:normality} would imply the validity of Conjecture~\ref{conj:mahler_generalized}, which is one generalization of a celebrated conjecture by Mahler from 1968.

Before doing so, we briefly recall the context in which Mahler's conjecture emerged. 
It relates to a classical and still largely unresolved problem: given two real numbers $x > 0$ and $\alpha>1$,  describe the distribution modulo $1$ of the sequence 
$(x\alpha^n)_{n\in \N}$.
A historical example, popularized by Mahler, and still poorly understood today, is the case $\alpha = 3/2$.

\begin{question}[asked to Mahler by a Japanese colleague] \label{question:mahler} Do there exist positive real numbers $x$ (called \emph{$Z$-numbers}) for which the sequence of fractional parts
	\[(\{x(3/2)^n\})_{n\in \N}\]
	is contained in the half interval $[0,1/2)$?
\end{question}

In 1968, Mahler conjectured that the answer is negative: $Z$-numbers do not exist. 
By an old theorem of Weyl \cite{Wey16}, it was already clear that the set of $Z$-numbers has Lebesgue measure $0$. Mahler furthermore proved that the set of $Z$-numbers is at most countable, and of  density $0$: the number of $Z$-numbers less than $x$ is $O(x^{0.7})$ (this bound was later improved by Flatto in \cite{Fla92}). His proof relies on the study of a class of binary words, which turn out to be exactly the minimal words in base $3/2$. 
At the time of writing, Question~\ref{question:mahler} remains unsolved, and many generalizations continue to be investigated (see, for example, \cite[chapter 3]{Bug12} for a survey until 2012, and \cite{Dub19} for a recent reference). 

\medskip

In the present article, we focus on a generalization of Question \ref{question:mahler}
that preserves its connection with minimal words.

\begin{question}\label{question:mahler_generalized} Let $p>q>1$ be coprime integers. Do there exist positive real numbers $x$ (called \emph{$Z_{p/q}$-numbers}) for which the sequence of fractional parts
	\[(\{x(p/q)^n\})_{n\in \N}\]
	is contained in the interval $[0,1/q)$?
\end{question}

We believe that there exists no $Z_{p/q}$-numbers when $p<q^2$.
\medskip

The non-existence of $Z_{p/q}$-numbers was already conjectured by Dubickas and Mossinghoff in the restricted case $1<q<p<q(q-1)$ \cite[Proposition~3.1]{DM09}. Along these lines, they establish after large-scale computations that there exist no $Z_{p/q}$-numbers smaller than $2^{57}, 10^{32}$, and $3^{42}$ for $p/q= 3/2$, $4/3$, and $5/3$, respectively (see \cite{DM09}, Theorem~5.1 and Tables~4 and 5 for more values of $p/q$).

\medskip

Moreover, we know that one can find infinitely many $Z_{p/q}$-numbers for every pair of coprime integers such that $p>q^2$. This result was established by Tijdeman in the particular case $q=2$ \cite[item (ii) p2]{Tij72}, and by Flatto in the general case \cite[Theorem 7.3, with $t=1/q$]{Fla92}.

\medskip

We now recall and prove our Theorem \ref{th:impliesMahler}.

\medskip

\noindent \textbf{Theorem \ref{th:impliesMahler} (reminder)}  
\emph{Let $p>q>1$ be coprime integers such that $p<q^2$. If all minimal words in rational base $p/q$ (except the one with seed the empty word) are normal over the alphabet $\{0,\dots,q-1\}$, then $Z_{p/q}$-numbers do not exist. } 

\begin{proof}
We assume that Conjecture~\ref{conj:normality} is true. Let $p>q>1$ be coprime integers such that $p<q^2$. We argue by contradiction. Let $x\in\R_{>0}$ be a $Z_{p/q}$-number. For every integer $n\geq 0$, we set $x_n:=x(p/q)^n$, $g_n:=\lfloor x_n\rfloor$ and $r_n:=\{x_n\}$. It follows from the definition that every $x_n$, for $n \in \N$, is also a $Z_{p/q}$-number. Therefore, without loss of generality, we may assume that $x \geq 1$. 

By Lemma~3.1 in \cite{DM09} (which is an immediate generalization of the work of Mahler in the case $p<q^2$), the sequences $(g_n)_n$ and $(r_n)_n$ fulfill the following relations: for every $n\in\N$,
\begin{equation}\label{eq:mots_de_malher}
g_{n+1}=\frac{pg_n+\alpha_n}{q} \hspace{0.5cm}\text{ and }\hspace{0.5cm} r_{n+1}=\frac{pr_n-\alpha_n}{q},
\end{equation}
where $\alpha_n\in\{0,\ldots,q-1\}$ is the smallest nonnegative integer such that $g_{n+1}$ is an integer, or, equivalently, such that $\alpha_n = -pg_n \mod q$. Therefore, by Lemma~\ref{lemma:link_wmin_operatorT}, since $g_0 = \lfloor x \rfloor \geq 1$, we have 
\[g_{n+1} = \Big\lceil \frac{p}{q}g_n\Big\rceil =\nmin(g_0,n). \]
Hence, by Remark~\ref{rk:how_to_compute_min_words}, 
\[\alpha_0\alpha_1\alpha_2\cdots = \wmin_{p/q}(g_0) = 
\wmin_{p/q}(\lfloor x \rfloor).\]
By Conjecture~\ref{conj:normality}, the infinite word $\wmin_{p/q}(\lfloor x\rfloor)\in\{0,\ldots,q-1\}^\N$ is normal. In particular, it must contain the subword $(q-1)(q-1)$ at least once. Let $m\in\N$ such that $\alpha_m=\alpha_{m+1}=q-1$. We thus have
\begin{equation}\label{eq:rm}
r_{m+1}=\frac{p}{q}r_m-\frac{q-1}{q} \quad\text{ and }\quad r_{m+2}=\frac{p}{q}r_{m+1}-\frac{q-1}{q}.
\end{equation}
Since $x$ is a $Z_{p/q}$-number, we also have  $r_m<1/q$, hence
\[
 0\leq r_{m+2}=\frac{p^2}{q^2}r_m-\frac{(p+q)(q-1)}{q^2}<\frac{p^2}{q^3}-\frac{(p+q)(q-1)}{q^2}.
\]
This yields
\begin{equation}\label{eq:pol_in_p}
p^2-q(q-1)p-q^2(q-1)>0.\end{equation}
Our aim is to prove that the inequality \eqref{eq:pol_in_p} is impossible under the condition $p<q^2$. Denote by $f(p)$ the left-hand side of \eqref{eq:pol_in_p}. The function $p\mapsto f(p)$ is a polynomial of degree $2$ with a positive discriminant. Denote by $r_1$ and $r_2$ its two roots, with $r_1\leq 0$ and $r_2\geq 0$.

We now prove that $q^2-r_2<1$. First, a direct computation shows that
\[
r_2=\frac{q(q-1)+\sqrt{q^2(q^2+2q-3)}}{2} \quad\text{ and }\quad q^2 = \frac{q(q-1)+\sqrt{q^2(q^2+2q+1)}}{2}.
\]
Hence
\[
	\begin{array}{rcl}
		q^2-r_2 &=& \ds\frac{q}{2}\big(\sqrt{q^2+2q+1} - \sqrt{q^2+2q-3}\,\big)\\
		&=& \ds\frac{2q}{\sqrt{q^2+2q+1} + \sqrt{q^2+2q-3}} < \frac{2q}{\sqrt{q^2}+\sqrt{q^2}} =1.
	\end{array}
\]
(In the last inequality, we used the estimate $2q-3\geq 0$, which is true since $q\geq2$.) The inequality $q^2-r_2<1$ is thus proven.

Finally, it follows from $r_1\leq 0$ that $ p \geq r_1 $, and from the double inequality $p<q^2$ and $q^2-r_2<1$, together with the fact that $p$ is an integer, that $ p\leq r_2$. Since 
$f$ is nonpositive between its roots, we have $f(p)\leq 0$, which contradicts \eqref{eq:pol_in_p}. The proof of Theorem~\ref{th:impliesMahler} is complete.
\end{proof}

We conclude this section with two remarks and a question.

\begin{remark} When $p>q^2$, there is no contradiction between the existence of $Z_{p/q}$-numbers and the expected normality of minimal words in rational base $p/q$. Indeed, under this condition, 
	the connection between $Z_{p/q}$-numbers and minimal words vanishes (the equalities \eqref{eq:mots_de_malher} in the proof of Proposition~\ref{th:impliesMahler} no longer hold). 
\end{remark}

\begin{remark} In the proof of Theorem \ref{th:impliesMahler}, we only marginally used the normality of minimal words. This leaves, in theory, considerable room for the possibility that Conjecture \ref{conj:mahler_generalized} is true while our Conjecture~\ref{conj:normality} is not. It also raises the following question.
\end{remark}

\begin{question}\label{question:further_mahler}
	Let $p>q$ be coprime integers such that $p<q^2$. Could the normality of minimal words in rational base $p/q$, if true, provide further insight into the distribution of the fractional parts $\{x(p/q)^n\}$, for $n \in \N$?  
\end{question}

\subsection{Connection with three other earlier conjectures}\label{sssect:linkwithconjectures}

Hereafter, we show that the validity of our normality Conjecture~\ref{conj:normality}  would imply the truth of two earlier conjectures: one by Akiyama (2008), concerning the existence of triple expansion of real numbers in rational base, and the other by Dubickas and Mossinghoff (2009), concerning the termination of certain iterated maps on integers.  
We also discuss the connection between our conjecture and the famous Collatz conjecture.

\paragraph{A conjecture by Akiyama.} 
To state the conjecture, it is useful to recall how positive real number are represented in rational base $p/q$. 
By definition, an infinite word $w=a_ka_{k-1}\dots a_0  a_{-1}a_{-2}\dots$ is an expansion of a positive real number $x$ if:
\begin{itemize}
	\vspace{-0.2cm}
	\item[-] all finite prefixes of $w$ belong to the language  $\{0^ku : k \in \N, u \in \lan_{p/q}\}$; or, equivalently, if after removing the eventual leading $0$s, the word $w$ labels an infinite path starting from the node $\eps$ in the tree $\tree_{p/q}$;
	\vspace{-0.2cm}
	\item[-] its valuation equals $x$:
	\[x=\frac{1}{q} \sum\limits_{-\infty}^{l=k} a_{l}\Big(\frac{p}{q}\Big)^{l}.\] 
\end{itemize}

As it is already the case for integer bases, for every rational base, there  exist countably many real numbers that admit several expansions. Akiyama's conjecture claims that these real numbers admit exactly two expansions. 

\begin{conjecture}[Akiyama, 2008, \cite{Aki08}]\label{conj:triple_expansions}
	Let $p>q$ be coprime integers. Every positive real number $x$ admits at most two expansions in rational base $p/q$.
\end{conjecture}

Multiple expansions in rational base are intimately connected to minimal and maximal words. Akiyama's conjecture is easily verified for $p\geq2q-1$: see \cite[Corollary 38]{AFS08}, 
or Remark \ref{rk:notripleexpansion} below. By contrast, in the case $p<2q-1$, the conjecture is believed to be difficult.
\medskip

We now prove that our normality conjecture, if true, would straightforwardly confirm Akiyama's conjecture.

\begin{proposition}\label{prop:impliesAkiyama}
	The validity of the normality Conjecture \ref{conj:normality} implies the truth of Akiyama's Conjecture \ref{conj:triple_expansions}.
\end{proposition}

\begin{proof}
	Let $p>q$ be coprime integers. It is a consequence of \cite[Section 5]{AFS08} that the existence of a real number admitting at least three expansions is equivalent to the existence of a word $w$ that is simultaneously a minimal and a maximal word.
	Since minimal words are written over the alphabet $\{0,\dots,q-1\}$
	and maximal words are written over the alphabet $\{p-q,\dots,p-1\}$, a single occurrence of the letter $0$ in a minimal word in base  $p/q$ is sufficient to prevent it from being a maximal word. Our normality conjecture, if true, obviously implies that the letter $0$ occurs in every minimal word. Therefore, no word can be simultaneously minimal and maximal.
\end{proof}

\begin{remark}\label{rk:notripleexpansion}
	Akiyama's conjecture, in the case $p\geq2q-1$, can easily be proven as follows. When $p>2q-1$, the two alphabets do not intersect, and therefore, no word can be simultaneously minimal and maximal, implying no triple expansion exists. When $p=2q-1$, the intersection of the two alphabets is the singleton $\{q-1\}=\{p-q\}$. Thus, a word that is simultaneously minimal and maximal must be written with one letter only and must be aperiodic (we recall from the Introduction that all minimal and maximal words are non-eventually periodic when $q\neq 1$). This is impossible.
\end{remark}

\paragraph{A conjecture by Dubickas and Mossinghoff.} In \cite[Section~1]{DM09}, the authors write the following question, whose positive answer they estimate as likely.

\begin{question}[Dubickas `4/3-problem']\label{question:DM09}
	Let $p,q$ be coprime integers such that $p>q>1$, and $S\subset\{0,\dots,q-1\}$ be a nonempty set. Is it true that the sequence of iterates of the map
	\[x \mapsto \begin{cases}
	\lceil px/q \rceil, \text{ if } x = s \bmod q, \text{ for some } s\in S,  \\ \text{STOP, otherwise.}
	\end{cases}
	\]
	terminates for any starting positive integer $x_0$?
\end{question}

\begin{proposition}\label{prop:implies43problem}
	If the normality Conjecture \ref{conj:normality} is true, then the answer to Question \ref{question:DM09} is `yes'.
\end{proposition}

\begin{proof}It is clear that the truth of our Conjecture~\ref{conj:equidistribution} (which states that every nonzero sequence of iterates for the operator $T_{p/q}=\lceil \frac{p}{q} \, \bm{\cdot} \rceil$ is equidistributed in the  residue classes modulo $q^l$ for every $l\geq 1$) implies that the answer to Question~\ref{question:DM09} is `yes'. Conjectures \ref{conj:normality} and \ref{conj:equidistribution} being equivalent by Theorem \ref{th:conj_equivalence}, the result follows.
\end{proof}

\begin{remark}\label{rk:DM09_conjecture} From a combinatorics on words perspective, Question~\ref{question:DM09} is equivalent to asking if all letters in $\{0,\dots,q-1\}$ appear in every minimal word in rational base $p/q$. 
\end{remark}

\paragraph{Collatz conjecture.}

Let us recall its statement.

\begin{conjecture}\label{conj:collatz}
	For every positive integer $x$, the sequence of iterates of the operator
	\[F : x \mapsto \begin{cases}
	\frac{3x+1}{2}, \text{ if $x$ is odd, }\\
	\frac{x}{2}, \text{ otherwise,}
	\end{cases}\]
	is eventually periodic with period $(1,2)$.
\end{conjecture}

The possibility of a link between minimal and maximal words in rational base $3/2$ (which we recall, first emerged in 1968 in the context of Mahler's conjecture) and the Collatz conjecture
has intrigued several authors throughout time (see, for example, \cite{Lag85},  \cite{DM09}, \cite{Dub09},\cite{Ais14}, and \cite{EVG25}). This intuition  relies on the \emph{similarity} between the Collatz operator $F$ and the operator $T_{3/2}$, which can be rewritten as
\[T_{3/2} : x \mapsto \begin{cases}
\frac{3x+1}{2}, \text{ if $x$ is odd, }\\
\frac{3x}{2}, \text{ otherwise.}
\end{cases}\]

It is interesting to reduce modulo $2$ the trajectories of positive integers $x$ under iterations of the Collatz map $F$, and compare the combinatorial properties of the infinite binary words thus obtained, which we will call \emph{Collatz words}, with those of minimal/maximal words in rational base $3/2$. Along this line, Dubickas established the same lower bound for the complexity of minimal words and that of Collatz words encoding divergent trajectories, if they exist \cite[Corollary~4 and Theorem~5]{Dub09}. More recently, Eliahou and Verger-Gaugry simultaneously studied the distribution of letters in Collatz words, and in the maximal word with seed the empty word \cite[Conjectures~5 and 18]{EVG25}. 

However, beyond these similarities, we are not aware, as of this writing, of any implication or, more generally, of any formal comparison regarding the relative difficulty between our Conjectures~\ref{conj:normality} and \ref{conj:equidistribution} and the Collatz conjecture.

\section{Numerical evidence supporting Conjecture~\ref{conj:normality}}\label{sect:numerical_evidence}

Minimal words have been the subject of numerous numerical experiments, notably by Mahler \cite{Mah68}, Flatto \cite{Fla92}, and Dubickas--Mossinghoff \cite{DM09}; however, they were looking for occurrences of specific subwords, and not investigating the presence and the distribution of all finite words.

\medskip

We recall that, due to Corollary \ref{cor:conj1_redundant}, it is sufficient to test Conjecture \ref{conj:normality} on minimal words only.

\subsection{Methodology}

\paragraph{Overview of the experiments and reproducibility.}

We conducted two series of experiments: one examining the
\emph{richness threshold} of minimal words in rational base $p/q$, and the other measuring how
much the distribution of their subwords deviates from uniformity. These  experiments cover three families of minimal words---representing over 40,000 words in total---which were carefully chosen to avoid exhibiting particular behaviors. The computations were performed using Python 3 and required several days on a standard laptop. The most time-consuming step is the computation of minimal words, due to the manipulation of large integers generated by iterations of the operator $T_{p/q}$. The public git repository \cite{Git26} contains both the functions and the datasets we used, together with complete tables and figures that could not be included in the paper.

\paragraph{Experiment 1: Computing the richness threshold of minimal words.}

 An infinite word $w \in \{0,\dots,q-1\}^{\N}$ is said to be \emph{rich} if its complexity is $\mathsf{p}_w(l)=q^l$ for all $l\geq 1$, that is, if it contains all the $q^l$ finite words of length $l$ as subwords (see \cite{Bug12}, p.91). Clearly, being rich is a prerequisite for being normal.
 
The \emph{richness threshold} of an infinite word $w$ is:
\[ \begin{array}{rccl} \rt_w:&\N&\longrightarrow&\N\cup\{\infty\}\\
        & l&\longmapsto& \inf\{L\in\N \text{ s.t. all words of length $l$ appear in $\pref_L(w)$}\}.
    \end{array}
\]
A word $w$ is rich if and only if its richness threshold function always takes finite values.

\begin{example}\label{ex:richness-threshold}
	The richness threshold of an infinite binary word beginning with
	\[w = 00010011001011101...\]
	takes the values $4, 8$ and $15$ for $l=1,2,$ and $3$, respectively.
\end{example}

In our first series of experiments, we compute the richness thresholds of minimal words for increasing $l$. We compare them with:
\begin{itemize}
    \vspace{-0.2cm}\item[-] the richness threshold of the expansions, in integer base $q$, of $\pi$ and $\sqrt{2}$ (which are widely conjectured to be normal $q$-ary words since the work of \cite{Bor50}),
    \vspace{-0.2cm}\item[-] the richness threshold of numerous \emph{random $q$-ary words} (i.e., an infinite word in which the letters are independently and uniformly drawn from $\{0,\dots, q-1\}$); it is well known that almost all such words are normal,
    \vspace{-0.2cm}\item[-] the quantity $q^l\log(q^l)$, which is the asymptotic value (as $l\rightarrow \infty$) of the expected richness threshold of random $q$-ary words \cite{mor87}.
\end{itemize}
We choose to compute the richness threshold of minimal words for the following reasons:
	\begin{itemize}
	 \vspace{-0.2cm} \item[-] Having finite richness thresholds  is a necessary condition for being normal.
	 \vspace{-0.2cm} 	 
	 \item[-] Furthermore, a low richness threshold suggests that all factors of smaller lengths appear reasonably often, which supports normality.
	 \vspace{-0.2cm} \item[-] The first values of the richness threshold of an infinite word $w$ can be computed definitively from a sufficiently long prefix of $w$, as illustrated in Example~\ref{ex:richness-threshold}.	 
\end{itemize}

\paragraph{Experiment 2: Measuring the deviation from uniformity.}

In our second series of experiments, we measure how much the distribution of subwords of length $l$ in growing prefixes of minimal words deviates from a uniform distribution. 
Specifically, for an infinite word $w \in \{0,\dots,q-1\}^{\N}$, and $l \in\N$, we define the \emph{length l deviation from uniformity} of $w$ by
\begin{equation}\label{eq:def_dfu}
    D_{w,l}(n):=\underset{v\in\{0,\dots,q-1\}^l}{\max}\,\Big|\frac{|\pref_n(w)|_v}{n-l+1}-\frac{1}{q^l}\Big|
\end{equation}
where $|\pref_n(w)|_v$ denotes the number of occurrences of $v$ in the prefix of length $n$ of $w$ (that is, the number of times $v$ appears when sliding a window of length $l$ along the  first $n$ letters of $w$). The intermediate quantity 
\[ \frac{|\pref_n(w)|_v}{n-l+1}\]
can thus be understood as the empirical frequency of the subword $v$ in $w$. 

Clearly, $w$  is normal if, and only if, for every $l \in \N$, we have $\lim_{n\rightarrow \infty} D_{w,l}(n) = 0$. In this case, the deviation from uniformity coincides with the notion of `discrepancy', as defined in \cite{Sch86}.

In our second series of experiments, we compute the length $l$ deviations from uniformity of minimal words, and compare them with those of numerous random $q$-ary words.

\paragraph{Computed minimal words.} Ideally, one would like to carry out the experiments on prefixes that are as long as possible, for as many minimal words as possible. However, due to time constraints, a compromise must be made between the length $n$ of the computed prefixes, the number of rational bases $p/q$ considered, and the number of seed words that we investigate. 
We thus choose to work with three families of minimal words. 
\begin{enumerate}
\item In the first family, the parameters $p$ and $q$ vary, while the seed word $u$ is fixed. More precisely, we studied all the words $\wmin_{p/q}(u)$ for $u=\rep_{p/q}(1) =q$, and $1<q<p<10$, where $p$ and $q$ are coprime. We computed the first one million letters of these 19 words.

\item In the second family, we focus on the bases $p/q = 3/2$, $7/2$, $8/3$, $11/3$, $8/5$, and $26/5$. We examine the minimal words generated by the following sets of twenty randomly chosen seed words whose valuations lie in $\{1,\ldots,1000\}$: \small
\[
    \left\{\begin{array}{c} \text{valuations of chosen} \\ \text{seed words for base $3/2$}\end{array}\right\} \;=\; \left\{\begin{array}{c} 97, 135, 159, 218, 224, 243, 258, 276, 382, 433, \\ 570, 604, 650, 670, 684, 771, 845, 870, 972, 990 \end{array}\right\}
\]

\textcolor{white}{.}
\vspace{-0.7cm}
\[
    \left\{\begin{array}{c} \text{valuations of chosen} \\ \text{seed words for base $7/2$}\end{array}\right\} \;=\; \left\{\begin{array}{c} 26, 115, 167, 190, 223, 243, 250, 255, 271, 294 \\ 316, 394, 408, 592, 763, 802, 804, 830, 885, 943 \end{array}\right\}
\]

\vspace{-0.5cm}
\[
    \left\{\begin{array}{c} \text{valuations of chosen} \\ \text{seed words for base $8/3$}\end{array}\right\} \;=\; \left\{\begin{array}{c} 33,108,188,336,342,458,470,579,596,631, \\ 641,670,767,785,805,849,883,916,958,1000 \end{array}\right\}
\]

\vspace{-0.5cm}
\[
\left\{\begin{array}{c} \text{valuations of chosen} \\ \text{seed words for base $11/3$}\end{array}\right\} \;=\; \left\{\begin{array}{c}  101, 201, 283, 289, 308, 310, 367, 409, 429, 439, \\ 760, 812, 817, 846, 891, 925, 929, 934, 939, 987 \end{array}\right\}
\]

\vspace{-0.5cm}
\[
    \left\{\begin{array}{c} \text{valuations of chosen} \\ \text{seed words for base $8/5$}\end{array}\right\} \;=\; \left\{\begin{array}{c} 61, 111, 116, 414, 432, 455, 477, 551, 592, 664 \\ 711, 749, 772, 791, 835, 856, 878, 945, 961, 965 \end{array}\right\}
\]

\vspace{-0.5cm}
\[
\left\{\begin{array}{c} \text{valuations of chosen} \\ \text{seed words for base $26/5$}\end{array}\right\} \;=\; \left\{\begin{array}{c}52, 121, 134, 186, 239, 248, 294, 390, 453, 505, \\ 519, 600, 670, 671, 795, 816, 878, 916, 917, 962\end{array}\right\}
\]

\normalsize
Again, we computed the first one million letters of these 120 minimal words.

\item In the third and final family, we focus on the four bases $p/q = 7/2$, $5/3$, $11/3$ and $6/5$, but study minimal words generated from a much larger number of different seed words. More precisely, we randomly selected 40,000 seed words whose valuations lie in $\{1,\dots,2^{50}\}$.  As a compromise, we computed the first 100,000 letters (instead of one million) of the corresponding minimal words. The list of the 40,000 random seeds is provided in the git repository \cite{Git26}.
\end{enumerate}
In total, we ran our two experiments on the first one million letters of $139$ minimal words, and on the first 100,000 letters of an additional 40,000 minimal words.

Finally, note that the bases $p/q$ studied in the second and third families of experiments have been chosen to include parameters $(p,q)$ such that $q(q-1)<p<q^2$, corresponding to a case not explored in \cite{DM09}, and   $p>q^2$, for which the $Z_{p/q}$-numbers conjecture does not hold (see Section \ref{sssect:implies_mahler}).

\paragraph{Generation of the data sets}

The random seed numbers used in Experiments 2 and 3 were generated using the Python \texttt{random.randint} function, which is based on the Mersenne Twister (MT19937) pseudorandom number generator. It is important to ensure that the seed is chosen from a sufficiently large set relatively to the sample size. Indeed, we wish to reduce the probability of choosing two seeds that lie on the same branch in the tree $\tree_{p/q}$ and would generate essentially the same minimal word.

The minimal words are then computed from the  seed numbers following Remark~\ref{rk:how_to_compute_min_words} (which itself stems from the MD algorithm). The procedure is described by Algorithm~\ref{alg:minword}. 

\begin{algorithm}[H]\label{alg:minword}
	\caption{Computation of the first $l$ letters of $\wmin_{p/q}(n)$}
	$nmin \gets$ $n$\;
	$wmin \gets \text{the empty word}$\;
	\For{$\text{cpt} \gets 1$ \KwTo $l$}{
		compute $loc, a$ as quotient and remainder of $-p \times nmin$ divided by $q$\;
		$nmin \gets -loc$\;
		$wmin \gets wmin \cdot a$; \qquad \qquad \# where $\cdot$ denotes the concatenation \\
	}
	\Return $wmin$\;
\end{algorithm}

The computation of minimal words is the bottleneck of our experiments, due to the manipulation of large integers. More precisely, computing the length-$l$ prefix of $\wmin_{p/q}(n)$ requires handling integers as large as $n(p/q)^l$. Consequently, computing the first one million letters of $\wmin_{9/2}(1)$ in the first family took $8$ minutes, while computing the $10{,}000$  minimal words in base $11/3$ in the third family took approximately $6$ hours on a standard laptop.


\subsection{Results for the richness threshold}

First, we display the richness thresholds of all binary minimal words in our first family.
 The results are gathered in Table~1 (which is divided into two parts due to space constraints).
 This table is to be read as follows: at the intersection of the row representing the word $w$ and the column representing the length $l$:
\begin{itemize}
\vspace{-0.1cm}\item[-] if the corresponding entry is positive, it is the richness threshold $\rt_w(l)$;
\vspace{-0.1cm}\item[-] if the entry is negative, its absolute value indicates how many words of length $l$ are missing in the prefix of length $10^6$ of $w$.
\end{itemize}
\vspace{-0.1cm}The seven penultimate rows describe the distribution (minimum, $5$th, $25$th, $50$th, $75$th and $95$th percentile, maximum) of the richness threshold for a family of 1,000 random binary words.
The last row gives the (asymptotic) expected value of the richness threshold for a  random binary word, namely  $2^l\log(2^l)$.

\bigskip


\begin{center}
	{\small
		\begin{tabular}{|c||c|c|c|c|c|c|c|c|c|c|}
			\hline
			$l$ & 1 & 2 & 3 & 4 & 5 & 6 & 7 & 8 & 9 & 10  \\
			\hline \hline
			$\wmin_{3/2}(\mathtt{2})$ & 2 & 6 & \textbf{51*} & 54 & 123 & 358 & 787 & 1479 & 2643 & 7272  \\
			\hline
			$\wmin_{5/2}(\mathtt{2})$ & 3 & 6 & \textbf{11*} & 52 & 221 & 228 & 661 & \textbf{992*} & 2589 & 6507  \\
			\hline
			$\wmin_{7/2}(\mathtt{2})$ & 2 & 8 & 34 & 86 & 115 & 201 & 905 & 1126 & 3160 & \textbf{5725*}  \\
			\hline
			$\wmin_{9/2}(\mathtt{2})$ & 4 & 7 & 29 & 42 & 128 & \textbf{188*} & 626 & \textbf{2365*} & \textbf{5589*} & 6548  \\
			\hline
			$\rep_2(\sqrt{2})$ & 2 & 10 & 19 & \textbf{22*} & 133 & 459 & 517 & 1806 & 3259 & 7185 \\
			\hline
			$\rep_2(\pi)$ & 3 & 5 & 20 & \textbf{25*} & 102 & 400 & 540 & 1351 & 3790 & 8034  \\
			\hline
			randwords min & 2 & 5 & 10 & 20 & 48 & 132 & 376 & 799 & 2164 & 5167 \\
			\hline
			5th centile & 2 & 5 & 12 & 30 & 75 & 192 & 456 & 1126 & 2574 & 5960 \\
			\hline
			25th centile & 2 & 6 & 17 & 41 & 100 & 241 & 568 & 1334 & 2937 & 6674\\
			\hline
			50th centile & 2 & 8 & 22 & 53 & 128 & 293 & 663 & 1529 & 3340 & 7510 \\
			\hline
			75th centile & 3 & 12 & 29 & 71 & 162 & 364 & 793 & 1792 & 3853 & 8411 \\
			\hline
			95th centile & 5 & 18 & 46 & 114 & 241 & 510 & 1108 & 2305 & 4806 & 10282 \\
			\hline
			randwords max & 13 & 34 & 104 & 193 & 454 & 1506 & 2363 & 4509 & 8677 & 18807 \\
			\hline
			$\lfloor 2^l\log(2^l)\rfloor$ & 1 & 5 & 16 & 44 & 110 & 266 & 621 & 1419 & 3194 & 7097  \\
			\hline
		\end{tabular}
		
		\vspace{0.2cm}
		\begin{tabular}{|c||c|c|c|c|c|c|c|}
			\hline
			$l$&  11 & 12 & 13 & 14 & 15 & 16 & 17 \\
			\hline \hline
			$\wmin_{3/2}(\mathtt{2})$ &18200& 39358 & 65137 & 154725 & 390091 & 821322 & $-63$ \\
			\hline
			$\wmin_{5/2}(\mathtt{2})$ &16605 & 31442 & 71030 & 189740 & 309169 & 827260 & $-64$ \\
			\hline
			$\wmin_{7/2}(\mathtt{2})$ & 21722& 41938 & 77728 & 208773 & 384796 & 894414 & $-60$ \\
			\hline
			$\wmin_{9/2}(\mathtt{2})$ & \textbf{23435*} &32075 & 81088 & 190265 & 358020 & 914320 & $-61$ \\
			\hline
			$\rep_2(\sqrt{2})$ & 18928 & 32231 & 83298 & 166437 & 396117 & 847032 & $-53$ \\
			\hline
			$\rep_2(\pi)$ & 17225 & 35851 & 71909 & 160119 & 405148 & 824328 & $-63$ \\
			\hline
			randwords min & 12121& 24725 & 56668 & 126747 & 264481 & 595077 & $-39$\\
			\hline
			5th centile & 13216& 29425 & 63624 & 139672 & 304475 & 651828 & $-51$\\
			\hline
			25th centile & 15056& 32737 & 71045 & 152830 & 331329 & 702967 & $-58$\\
			\hline
			50th centile & 16566& 35726 & 76342 & 163915 & 353742 & 746354 & $-63$\\
			\hline
			75th centile & 18438 & 38913 & 82981 & 178882 & 384375 & 808695 & $-69$\\
			\hline
			95th centile & 22111& 46157 & 97283 & 210337 & 445093 & 915304 & $-77$\\
			\hline
			randwords max & 35662& 58696 & 151144 & 296020 & 817548 & $-1$ & $-93$\\
			\hline
			$\lfloor 2^l\log(2^l)\rfloor$ &15615 & 34069 & 73817 & 158991 & 340695 & 726817 & 1544487\\
			\hline
	\end{tabular}}
	\normalsize
	
	\medskip
	\noindent\textbf{Table 1:} Richness thresholds for $q=2$ (due to space constraints, the table is divided into two parts).
\end{center}


\paragraph{Analysis of Table~1.}

First, we observe that all binary subwords of length $16$ appear in the prefix of length 1,000,000 of the four minimal words considered. These prefixes miss between $60$ and $64$ subwords of length $17$. By comparison, the binary expansions of $\pi$ and $\sqrt{2}$ miss $53$ and $63$ subwords of length $17$, respectively, while $50\%$ of our dataset of random binary words of length 1,000,000 miss between $58$ and $69$ subwords of length $17$. Therefore, as far as we can tell, our measures support the hypothesis that these minimal words are rich.

It is interesting to compare more closely the richness thresholds of our minimal words with those of our dataset of random words. In this direction, we observe that all 64 computed richness thresholds for minimal words lie between the extremal values taken by the 1,000 random words. Among these 64 values, four of them take a value smaller than that of $95\%$ of random words, and four others take a value larger than that of $95\%$ of random words. These values are highlighted in the table.  Among them, we checked that the entry $\rt_{\wmin_{5/2}}(8)=992$ is the one that differs the most from the statistical richness thresholds of random words: only $0.7\%$ of the random words we considered have a smaller richness threshold. However, this ``anomaly'' disappears for $l \geq 8$, as far as we can measure. 
 Similarly, the three values $\rt_{\wmin_{9/2}}(8)$, $\rt_{\wmin_{9/2}}(9)$, and $\rt_{\wmin_{9/2}}(11)$ are larger than the $95$th percentile, but the intermediate and subsequent values $\rt_{\wmin_{9/2}}(10)$ and $\rt_{\wmin_{9/2}}(13)$ lie below the $25$th percentile.
It is a general observation: as far as we can see in Table~1, the richness threshold function of the minimal words considered does not appear to differ substantially from that of random binary words. This could support Conjecture~\ref{conj:normality}.

\bigskip

For the sake of readability, we omit the tables for the remaining 15 minimal words in our first family. These tables are available in the public Git repository~\cite{Git26}. Their analysis leads to conclusions similar to those drawn from Table~1.


\bigskip

In Table~2, we present the richness thresholds of the minimal words in our second family (that is, twenty randomly chosen minimal words, computed up to length $10^6$) in the case of base $8/3$. For the sake of readability, the results for bases $3/2$, $7/2$, and $8/5$ are omitted, but provided in the git repository.


\begin{center}
	{\small
		\begin{tabular}{|c||c|c|c|c|c|c|c|c|c|c|c|}
			\hline
			$l$ & 3 & 4 & 5 & 6 & 7 & 8 & 9 & 10 & 11 \\
			\hline \hline
			$\wmin_{8/3}(u_1)$ & 101 & 324 & 1467 & 5186 & 15357 & 61842 & 196505 & 794699 & $-638$\\
			\hline
			$\wmin_{8/3}(u_2)$  & 131 & 445 & 1513 & 5865 & \textbf{13673*} & 60611 & 201995 & 586281 & $-660$\\
			\hline
			$\wmin_{8/3}(u_3)$  & 79 & 376 & 1890 & 4379 & \textbf{30884*} & 62889 & 207741 & 675919 & $-608$\\
			\hline
			$\wmin_{8/3}(u_4)$  & 125 & 459 & 2036 & 5461 & 16702 & 55581 & 196247 & 646705 & $-642$\\
			\hline
			$\wmin_{8/3}(u_5)$  & \textbf{219*} & 551 & 1734 & 5119 & 20158 & 69743 & 192828 & 699757 & $-608$\\
			\hline
			$\wmin_{8/3}(u_6)$  & 112 & 368 & 1420 & 4054 & 21026 & 67403 & \textbf{170275*} & 706085 & $-628$\\
			\hline
			$\wmin_{8/3}(u_7)$  & 67 & 296 & 1332 & 6259 & 16860 & 68380 & 196740 & 650223 & $-608$\\
			\hline
			$\wmin_{8/3}(u_8)$  & 102 & 302 & 1435 & 5665 & 14792 & 64001 & \textbf{261771*}& 630669 & $-595$\\
			\hline
			$\wmin_{8/3}(u_9)$  & 64 & 339 & 1684 & 5916 & 15667 & 61957 & 186713 & 604557 & $-634$\\
			\hline
			$\wmin_{8/3}(u_{10})$  & \textbf{220*} & 399 & 1403 & 4792 & 19613 & 67004 & \textbf{163180*} & 655187 & $\mathbf{-681}\textbf{*}$\\
			\hline
			$\wmin_{8/3}(u_{11})$  & 118 & 551 & 1376 & 5161 & 19444 & 65987 & \textbf{165643*} & 670311 & $-613$\\
			\hline
			$\wmin_{8/3}(u_{12})$  & 98 & 327 & 1145 & \textbf{7172*} & 18647 & 63339 & 186010 & 589654 & $-632$\\
			\hline
			$\wmin_{8/3}(u_{13})$  & \textbf{46*} & 264 & 1320 & 6031 & 16866 & 55016 & 208499 & \textbf{871809*} & $-660$\\
			\hline
			$\wmin_{8/3}(u_{14})$  & 103 & 407 & 1657 & 5582 & 19505 & 55349 & 291635 & 720286 & $-665$\\
			\hline
			$\wmin_{8/3}(u_{15})$ & 151 & 408 & 1441 & \textbf{3791*} & 16716 & 61525 & 208025 & 642467 & $-624$\\
			\hline
			$\wmin_{8/3}(u_{16})$ & 82 & 258 & 1085 & 5424 & 20403 & \textbf{76416*} & 232390 & 632903 & $-642$\\
			\hline
			$\wmin_{8/3}(u_{17})$  & 135 & \textbf{944*} & 1811 & 5536 & \textbf{13357*} & 58023 & 222863 & \textbf{885517*} & $-629$\\
			\hline
			$\wmin_{8/3}(u_{18})$  & 112 & 405 & 1607 & 4943 & 18437 & 61846 & 228336 & 750062 & $-604$\\
			\hline
			$\wmin_{8/3}(u_{19})$  & 91 & 535 & 1223 & 4607 & \textbf{23991*} & \textbf{49271*} & 215391 & 799419 & $-608$\\
			\hline
			$\wmin_{8/3}(u_{20})$  & 121 & \textbf{663*} & 1944 & 4639 & 18637 & 61147 & 243133 & 608729 & $-588$\\
			\hline
			$\rep_3(\sqrt{2})$  & 66 & 377 & 1290 & \textbf{7404*} & 16511 & 56260 & 211187 & 790264 & $-629$ \\
			\hline
			$\rep_3(\pi)$  & 119 & 348 & 1978 & 6379 & 15779 & \textbf{79122*} & 183178 & 584098 & $-647$ \\
			\hline
			rand min  & 45 & 175 & 892 & 3310 & 12777 & 43294 & 153480 & 522425 & $-530$\\
			\hline
			5th centile  & 61 & 269 & 1063 & 3991 & 14329 & 50780 & 173011 & 581477 & $-583$\\
			\hline
			25th centile  & 81 & 327 & 1265 & 4579 & 16119 & 55641 & 188647 & 631602 & $-610$\\
			\hline
			50th centile  & 102 & 391 & 1444 & 5126 & 17606 & 60139 & 202943 & 673103 & $-626$\\
			\hline
			75th centile & 129 & 462 & 1675 & 5791 & 19581 & 65387 & 219601 & 729835 & $-642$\\
			\hline
			95th centile  & 185 & 608 & 2066 & 6940 & 23656 & 76001 & 257283 & 828424 & $-667$\\
			\hline
			rand max  & 293 & 1006 & 3546 & 11360 & 31557 & 115515 & 325881 & $-1$ & $-717$\\
			\hline
			$\lfloor 3^l\log(3^l)\rfloor$ & 88 & 355 & 1334 & 4805 & 16818 & 57663 & 194615 & 648719 & \tiny{2140774} \\
			\hline
	\end{tabular}}
	\normalsize
	
	\medskip
	\noindent\textbf{Table 2:} Richness thresholds in base $p/q=8/3$ for 20 randomly chosen seed $n\in\{1,2,\ldots,1000\}$. The columns $l=1,2$ have been removed due to space constraint.
\end{center}


\paragraph{Analysis of Table~2.} As in Table~1, the richness thresholds of the 20 random minimal words in rational base $8/3$ appear to take finite values, as far as we can measure, thus supporting that our richness hypothesis might be valid for every minimal words, regardless of the seed. Moreover, these values again appear to lie within the range of those taken by our dataset of 1,000 random ternary words of length 1,000,000. Minor isolated deviations appear (anomalous data are highlighted), but there is no evidence of behavior substantially different from that of random words.

\bigskip

Finally, we present the richness thresholds of our third (large) family of minimal words in graphical form for the four rational bases considered.

\bigskip

\noindent
\begin{minipage}[c]{.5\linewidth}
    \includegraphics[scale=0.24]{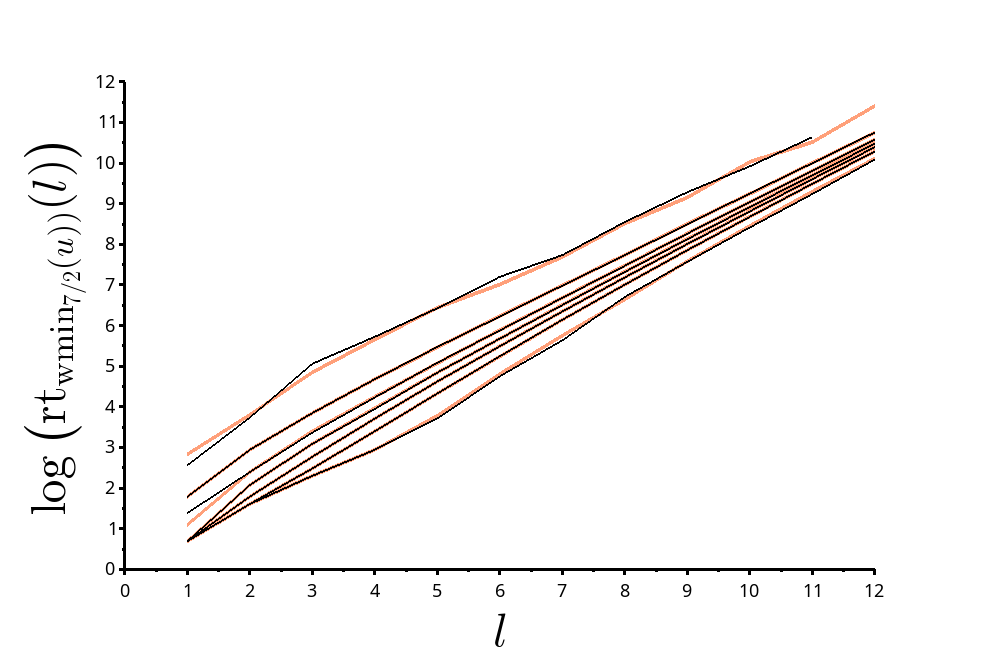}
\end{minipage}
\hfill
\begin{minipage}[c]{.5\linewidth}
    \includegraphics[scale=0.24]{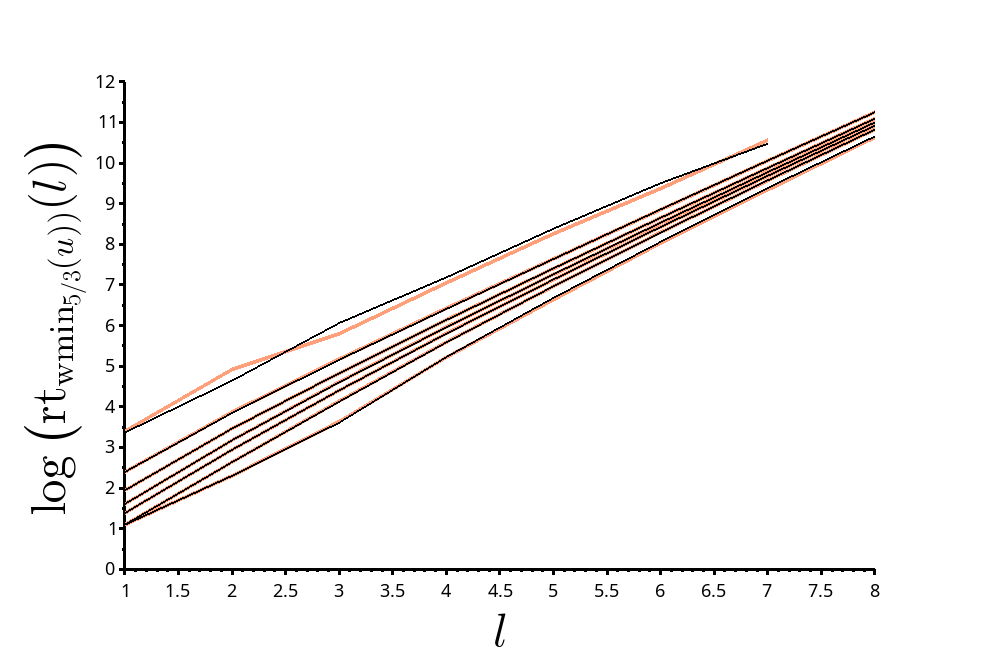}
\end{minipage}

\noindent
\begin{minipage}[c]{.5\linewidth}
    \includegraphics[scale=0.24]{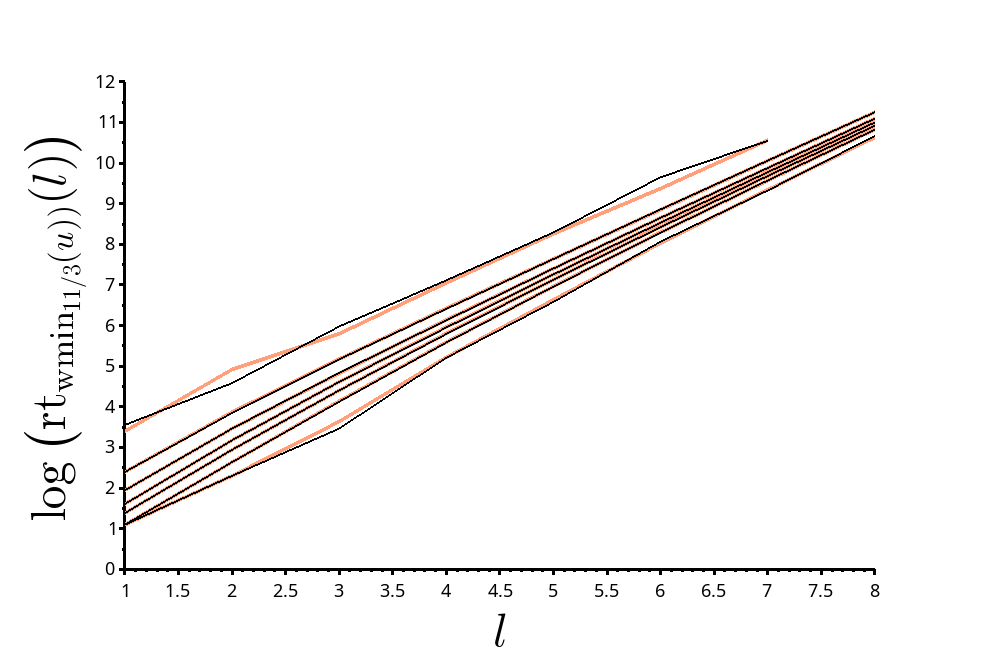}
\end{minipage}
\hfill
\begin{minipage}[c]{.5\linewidth}
    \includegraphics[scale=0.24]{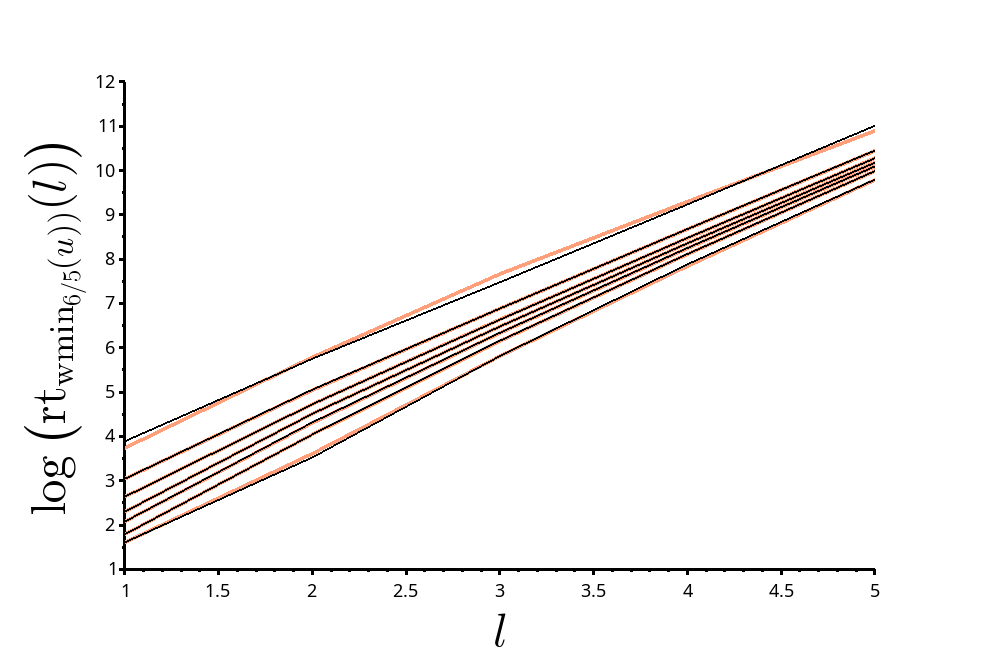}
\end{minipage}

\begin{center}
\noindent\textbf{Figure 2:} Richness thresholds for $10,000$ randomly chosen minimal words in base $7/2$ (top left), $5/3$ (top right), $11/3$ (bottom left), and $6/5$ (bottom right).
\end{center}

Each panel of Figure~2 compares the statistical properties of the richness thresholds of 10,000 randomly chosen minimal words (in black) with those of 10,000 random $q$-ary words (in red or gray). All the words were computed up to length 100,000. From bottom to top, we display the minimum, the 5th, 25th, 50th, 75th, and 95th percentiles, and the maximum of these two sets of words.

\paragraph{Analysis of Figure~2} As far as we can see, each of the 40,000 richness threshold functions of the minimal words in our dataset takes finite values, except sometimes for the last value of $l$ due to the limited length of the computed prefix. This supports our hypothesis that all minimal words are rich. Moreover, the growth rate of the richness threshold functions of minimal words appears similar to that of random words, and the distribution of the values taken by the richness threshold appears, as far as our division into percentiles allows us to observe, similar to that of random words.

To confirm this latter observation, we performed eight ``cut views'' (two for each base). Figure~3 compares the cut views for $l=5$ and $l=10$, that is, the distributions of $\rt_w(5)$ and $\rt_w(10)$ among minimal words in base $7/2$ in our dataset with those of random binary words. In both cases, the distributions of the richness thresholds for minimal and random words are very similar and exhibit the same Poisson-like shape.

\bigskip

\noindent \begin{minipage}[c]{.5\linewidth}
	\includegraphics[scale=0.41]{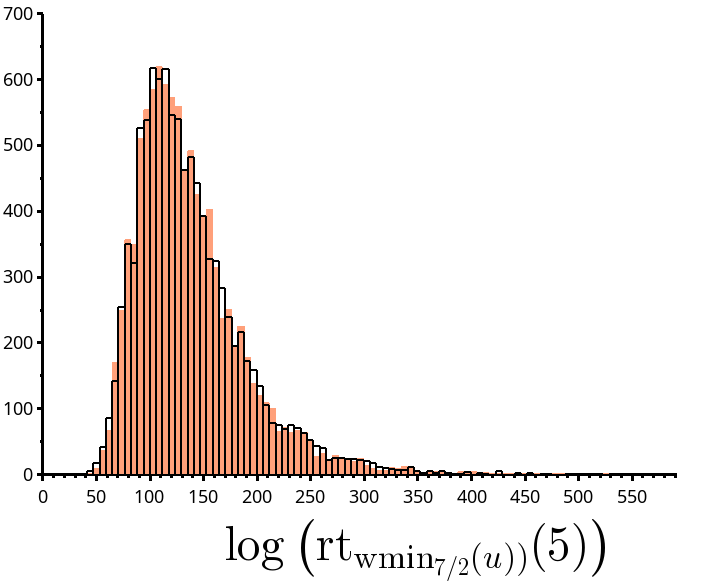}
\end{minipage}
\begin{minipage}[c]{.5\linewidth}
	\includegraphics[scale=0.41]{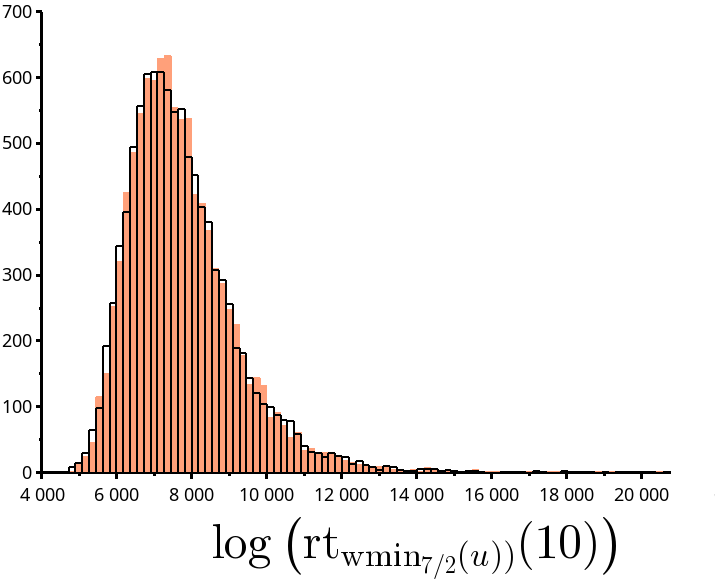}
\end{minipage}

\begin{center}
	\noindent\textbf{Figure~3:} Distribution of  $\rt_w(5)$ (left) and $\rt_w(10)$ (right) for 10,000 randomly chosen minimal words in base $7/2$ (represented by a black curve) and 10,000 random binary words (represented by a red or gray area).
\end{center}

%

\subsection{Results for the deviation from uniformity}

We begin by displaying the length-$l$ deviation of a single minimal word to visualize the shape of the function.

\bigskip
In Figure~4, both the horizontal and vertical scales are logarithmic. The black curve (the more ``wavy'' one) represents the length-$7$ deviation of the minimal word $\wmin_{7/2}(\mathtt{2})$, that is, the function
\[n \mapsto D_{\wmin_{7/2}(\mathtt{2}),7}(n),\]
up to $n = 10^6$, with a step of $1000$ (which helps accelerate the computation and smooth the figure). It is compared with the statistical properties of random binary words: the five red (or gray) curves (which appear ``parallel'') represent, from bottom to top, the minimum, the first decile, the mean, the ninth decile, and the maximum length-$7$ deviations, still with a step of $1000$, computed from a set of 1,000 random binary words.

\begin{figure}[!h]
\begin{center}
\includegraphics[scale=0.25]{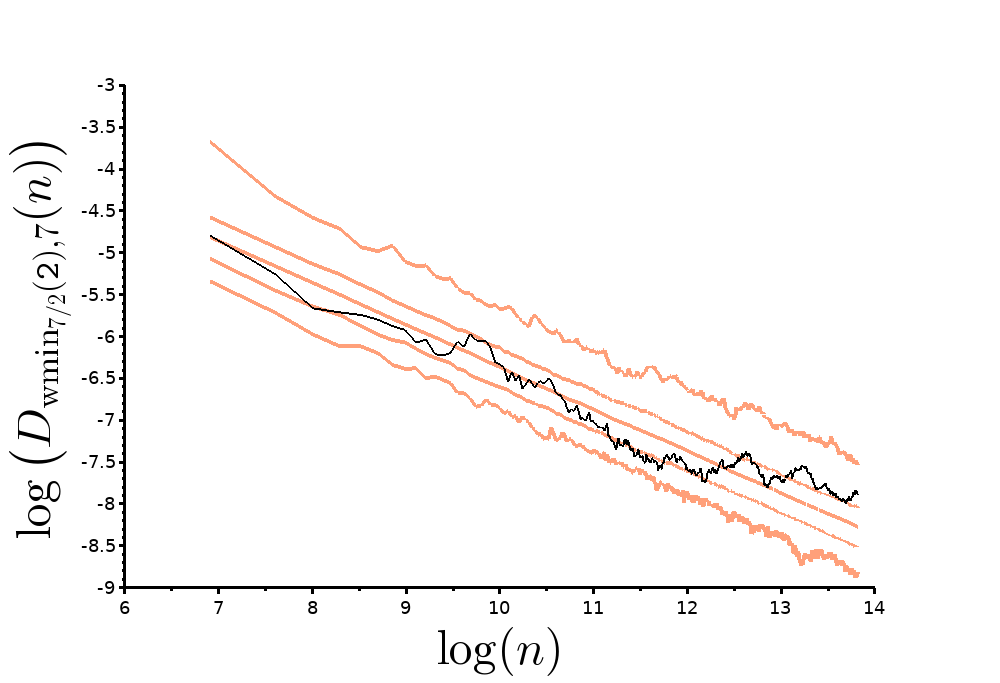}

\noindent\textbf{Figure~4:} Deviation from uniformity for subwords of length $l=7$,\\ in the minimal word $\wmin_{7/2}(\mathtt{2})$.
\end{center}
\end{figure}

Note that when studying the deviation from uniformity, the parameter $l$ (here $l = 7$) must be carefully chosen, depending on $q$ (the alphabet size) and the maximum  prefix length we consider. On the one hand, it is preferable to choose $l$ as large as possible: if the distribution of subwords of length $l$ is close to uniformity, then smaller subwords should also exhibit similar behavior. 
On the other hand, if $l$ is too large compared to the maximum prefix length (which is limited by the computational capacity of the machine), the quantity $D_{w,l}(n)$ loses its ability to accurately capture---and thus compare---the empirical frequency.

\paragraph{Analysis of Figure~4.} As can be observed:\\
1) The black curve decreases, indicating that the subwords of length $7$ become more and more fairly distributed in $\wmin_{7/2}(\mathtt{2})$ as the considered prefix grows. \\
2) The deviation $n \mapsto D_{\wmin_{7/2}(\mathtt{2}),7}(n)$ appears to decrease at a rate comparable to that of random binary sequences, which is asymptotically given by \cite[Theorem~1]{Phi75}:
\begin{equation}\label{eq:iteratedlog}
    \mathcal{O}\Big(\frac{\sqrt{\log\log(n)}}{\sqrt{n}}\Big).
\end{equation}

These two observations are confirmed by Figures~5 and 6. In Figure~5, we display, as a point cloud, the same function (that is, the deviation from uniformity for subwords of length $7$), still up to $n = 1{,}000{,}000$ and with a step of $1000$, for 20 randomly chosen minimal words in base $7/2$ from our second family. We continue to compare them with the statistical properties of random binary words.

\begin{center}
\includegraphics[scale=0.25]{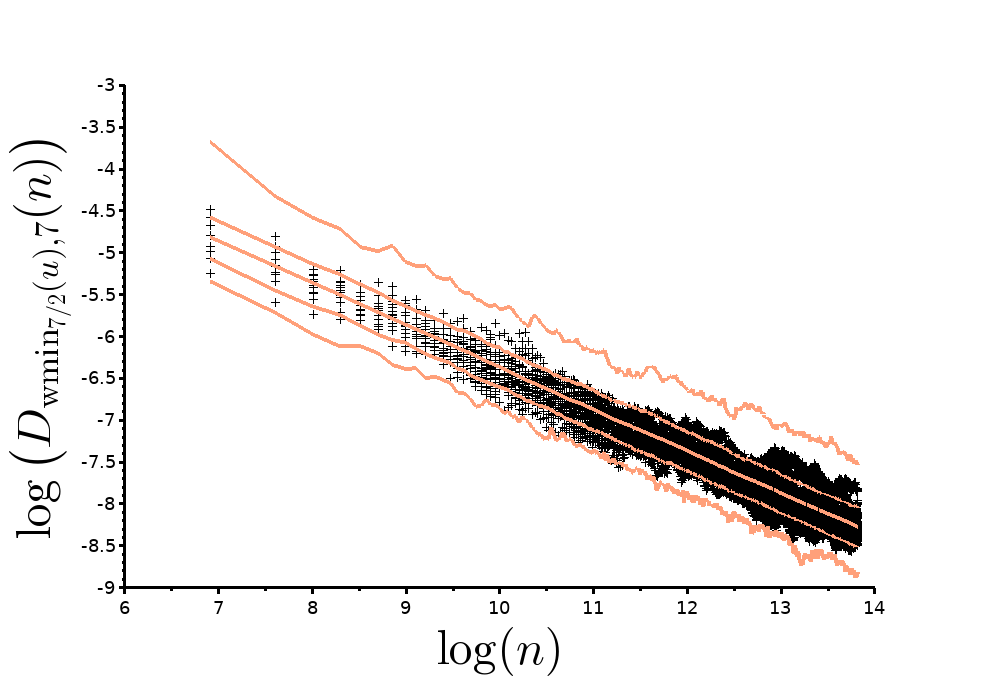}

\noindent\textbf{Figure~5:} Deviation from uniformity for subwords of length $7$ in minimal words $\wmin_{7/2}(u)$ obtained from $20$ randomly chosen seed words $u$.
\end{center}

In Figure~6, we display the deviations from uniformity for our third family (consisting of 10,000 words analyzed up to length $10^5$) for the four rational bases $7/2$, $5/3$, $11/3$, and $6/5$. For readability, we plot only their statistical properties (in black)---namely, the minimum, first decile, mean, ninth decile, and maximum---and compare them with those (in red, or gray) of 10,000 random words of length $10^5$.

\bigskip

\noindent
\begin{minipage}[c]{.5\linewidth}
    \includegraphics[scale=0.24]{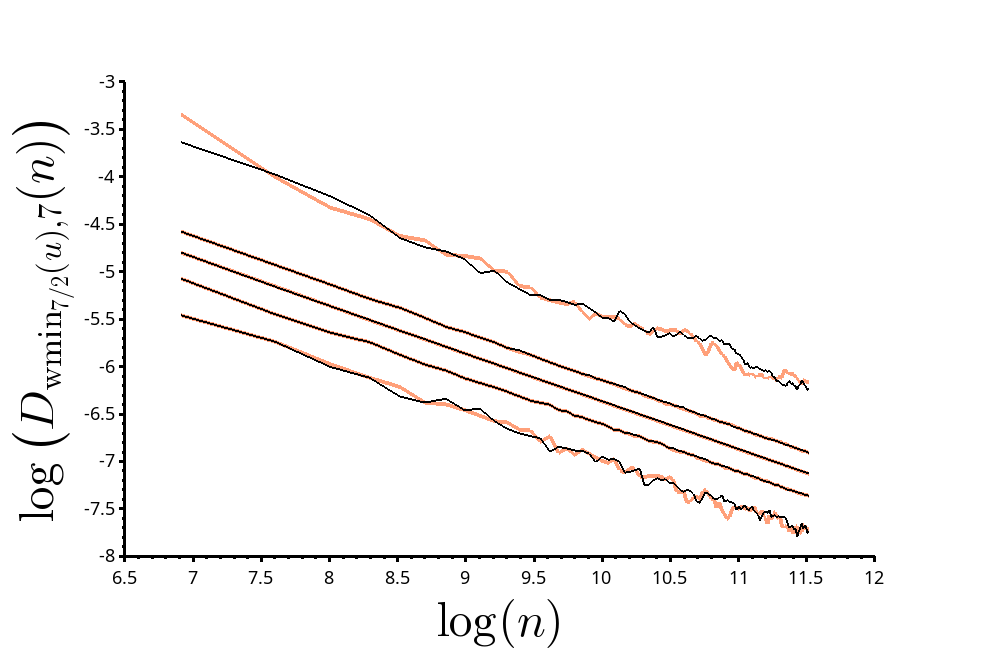}
\end{minipage}
\hfill
\begin{minipage}[c]{.5\linewidth}
    \includegraphics[scale=0.24]{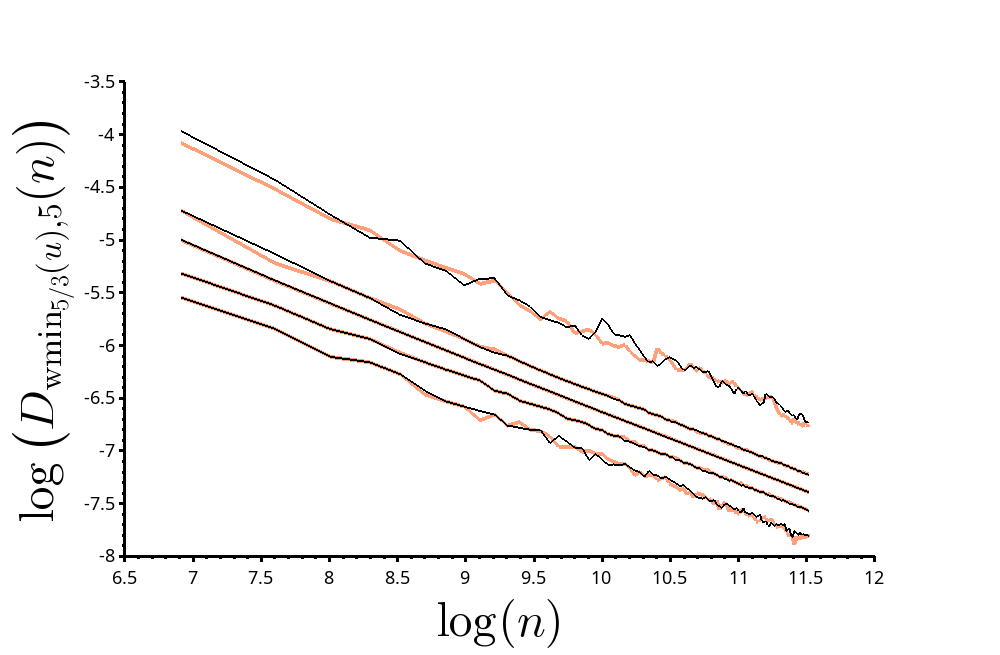}
\end{minipage}

\noindent
\begin{minipage}[c]{.5\linewidth}
    \includegraphics[scale=0.24]{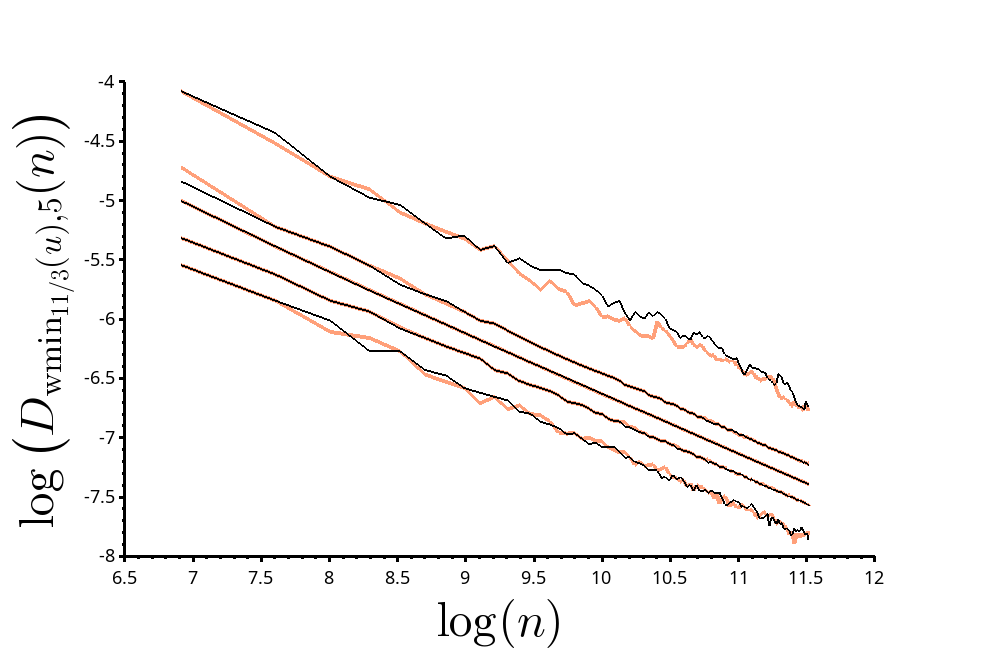}
\end{minipage}
\hfill
\begin{minipage}[c]{.5\linewidth}
    \includegraphics[scale=0.24]{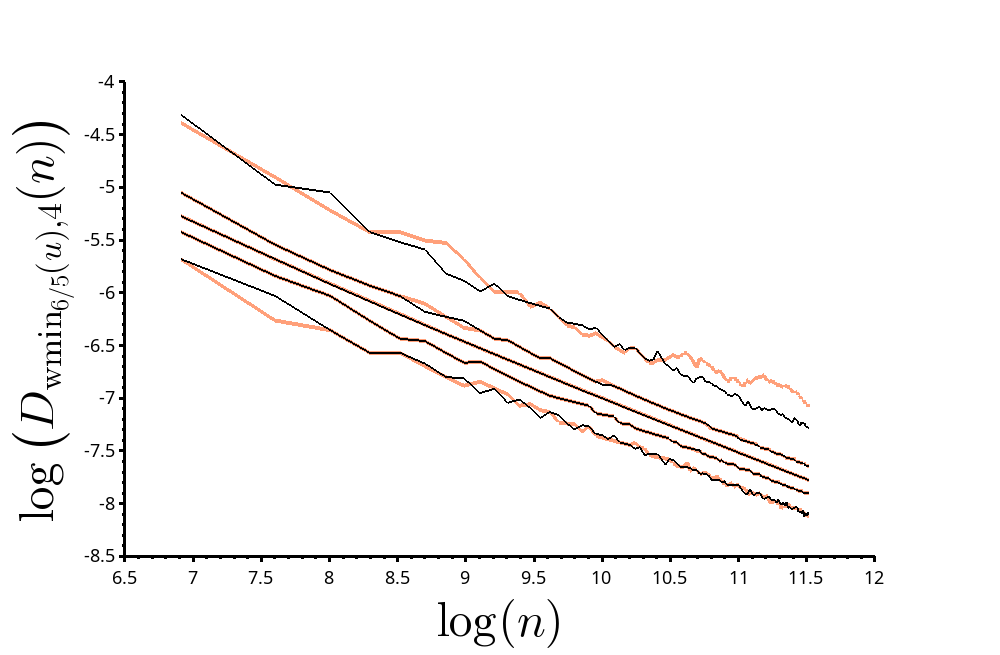}
\end{minipage}

\begin{center}
\noindent \textbf{Figure~6:} Comparison between the statistical properties of 10,000 minimal words in bases $p/q=7/2$ (top left), $5/3$ (top right), $11/3$ (bottom left) and $6/5$ (bottom right), and those of random $q$-ary words.
\end{center}

\medskip

As can be observed in the four panels of Figure~6, the distributions of the deviations from uniformity for minimal words and for random words  appear similar, as far as our division into quartiles allows us to assess. To confirm this observation, we performed three ``cut views'' ($n=3000$, $n=15000$ and $n=100000$) for each base.

\noindent \begin{minipage}[c]{.5\linewidth}
	\includegraphics[scale=0.235]{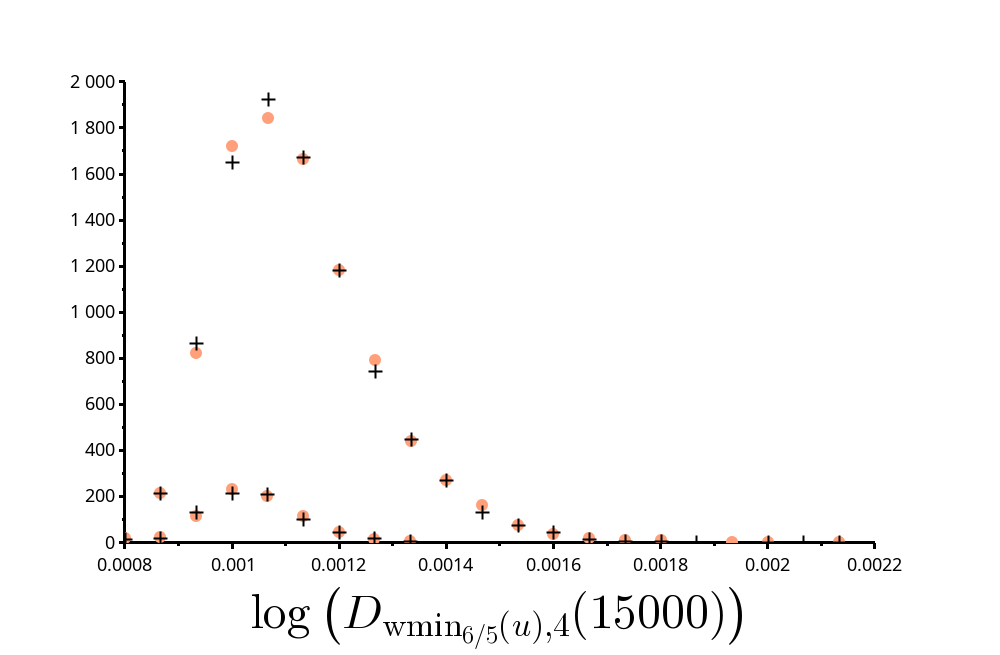}
\end{minipage}
\begin{minipage}[c]{.5\linewidth}
	\includegraphics[scale=0.235]{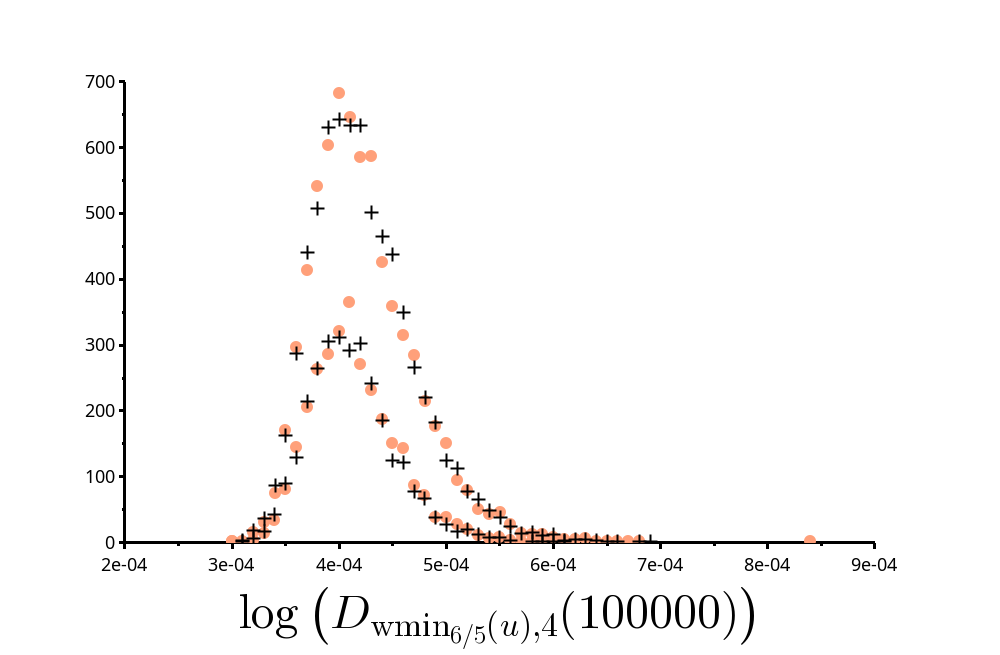}
\end{minipage}

\begin{center}
	\noindent\textbf{Figure~7:} Distribution of  $D_{w,4}(15000)$ (left) and $D_{w,4}(100000)$ (right) for 10,000 randomly chosen minimal words in base $6/5$ (represented by black crosses) and 10,000 random $5$-ary words (represented by red or gray dots).
\end{center}
Figure~7 compares the distributions of $D_{w,4}(15{,}000)$ and $D_{w,4}(100{,}000)$ for the 10,000 minimal words in base $6/5$ in our dataset with those for 10,000 random $5$-ary words. In both cases---as in all the other cut views we performed (see the public git repository \cite{Git26})---the distributions of the deviations for minimal and random words appear very similar and exhibit a superposition of two Poisson-like shapes. We verified that this superposition originates from the absolute value in the definition of the deviation from uniformity (see Expression~\eqref{eq:def_dfu}). More precisely, the distributions of the negative contributions to $D_{w,l}(15{,}000)$ and $D_{w,l}(100{,}000)$ form the smaller wave in both cases, while the positive contributions form the larger wave.

\subsection{Conclusion and perspectives}

We studied the richness threshold and deviation from uniformity in numerous minimal words to test whether our hypotheses of richness (which is a prerequisite for normality) and normality are credible. Beyond appearing merely normal, in all the cases we studied---and as far as our observations go---minimal words in base $p/q$ exhibit behavior indistinguishable from that of random $q$-ary words. From this perspective, they differ significantly from the two most-studied families of normal words: the $q$-ary Champernowne word and the infinite de Bruijn words. This difference is already clearly visible in Figure~8.
The $q$-ary Champernowne word (for $q\geq 2$) is defined as the concatenation of the base-$q$
 expansions of all positive integers, in increasing order \cite{Cha33}. An infinite $q$-ary de Bruijn word is an infinite word $w$ in which every $q$-ary finite word of length $l$ occurs exactly once in the prefix of length $q^l+l-1$; such words exist for every $q\geq 3$ (see \cite{BH11} for a complete proof) and are normal \cite{Uga00}.

\noindent
\begin{minipage}[c]{.5\linewidth}
    \includegraphics[scale=0.24]{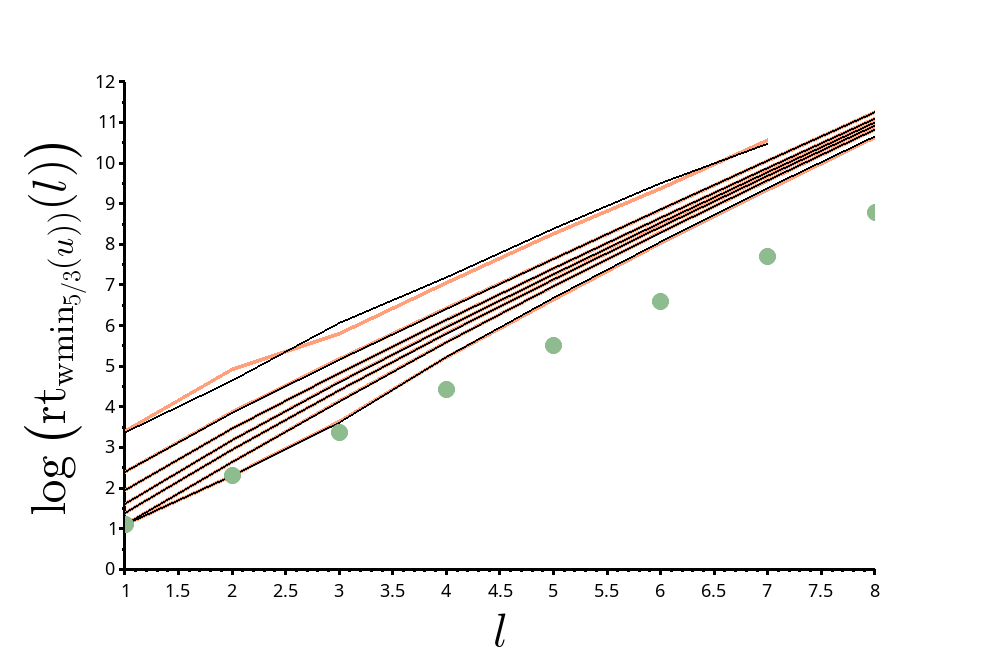}
\end{minipage}
\hfill
\begin{minipage}[c]{.5\linewidth}
    \includegraphics[scale=0.24]{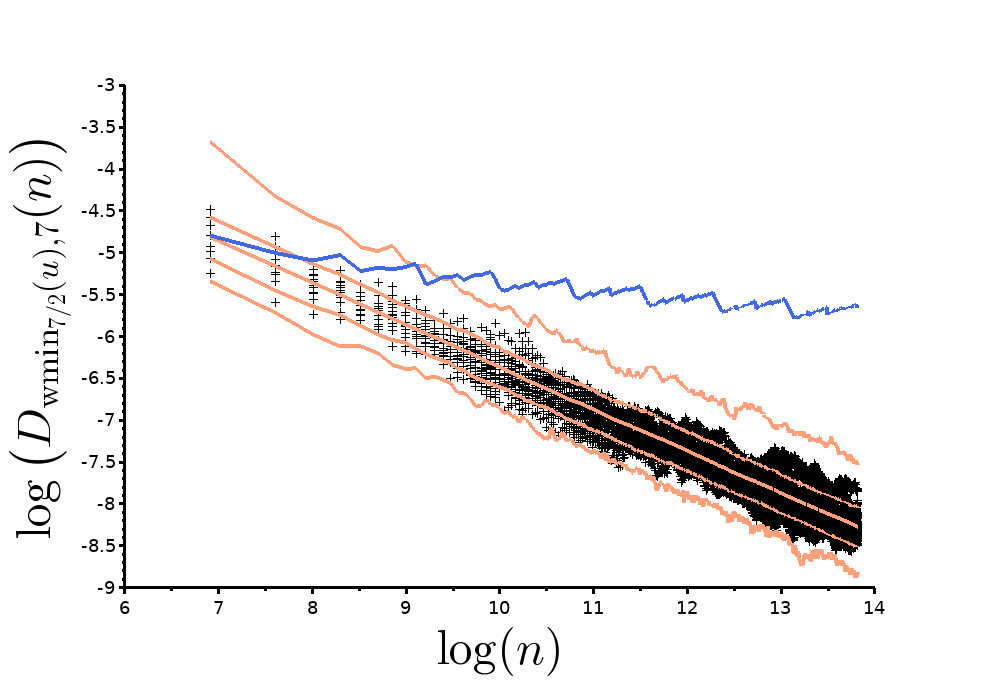}
\end{minipage}

\noindent \textbf{Figure~8:} Left: Top-right panel of Figure~2, in which we additionally plot, with green (or dark gray) large dots, the richness threshold of a ternary de Bruijn word. Right: Figure~5, where we additionally show, in blue (or dark grey), the discrepancy of the binary Champernowne word.

\bigskip
Note that the richness thresholds of the $q$-ary Champernowne word $w_{Cq}$ and any infinite $q$-ary de Bruijn word $w_{Bq}$ are easy to calculate:  
\[\rt_{w_{Cq}}(l)= lq^l - \frac{q^l-1}{q-1}+l+1, \quad \text{and} \quad rt_{w_{Bq}}(l)=q^l+l-1. \]
(For Champernowne, this follows from the fact that the subword $0\dots0$ is always the last $q$-ary word of length $l$ to appear, and it first appears in the base-$q$ expansion of $q^l$. For de Bruijn words, this result follows directly from their definition.)
The discrepancy of the Champernowne word is estimated in \cite{Sch86} (see also \cite{BG24}). Investigations related to the discrepancy of infinite de Bruijn words are presented in \cite{ABM24}.

\begin{question}
	Do minimal words in rational base $p/q$ appear to satisfy some supernormality properties, for instance, Poisson genericity \cite{ABM22}? Conversely, does there exist an interesting property that is satisfied by almost all $q$-ary words but not by minimal words?
\end{question}

It would also be interesting to understand the exact scope of our conjecture.

\begin{question}
	Does our Conjecture \ref{conj:normality} still hold for generalizations of rational base number systems \cite{AkiDraft,Ros25}?
\end{question}

 


\paragraph{Acknowledgements} The first author would like to express her gratitude to Ver\'onica Becher  and Shigeki Akiyama  for  sharing insights about supernormality properties and multiple expansions in rational bases. The authors thank the anonymous referees for their careful reading and numerous constructive comments, which greatly improved  the paper.

\noindent This work was supported by the Center for Mathematical Modeling (CMM) BASAL fund FB210005 for center of excellence from ANID-Chile.


{\small
\bibliography{biblio}
\bibliographystyle{alpha}}

\end{document}